\documentclass{amsart}
\usepackage{amsmath,amssymb,amscd,amsthm,latexsym}

\usepackage{mathrsfs}

\usepackage[all]{xy}

\usepackage{graphicx}

\usepackage[usenames, dvipsnames]{color}

\definecolor{NTNUblue}{RGB}{0,80,158}
\definecolor{NTNUbluesupport}{RGB}{62,98,138}
\definecolor{NTNUorange}{RGB}{239,129,20}

\usepackage{hyperref}

\usepackage[normalem]{ulem}

\newcommand{\bi}{\begin{itemize}}
\newcommand{\ei}{\end{itemize}}
\newcommand{\bn}{\begin{enumerate}}
\newcommand{\en}{\end{enumerate}}

\newcommand{\bq}{\begin{equation}}
\newcommand{\eq}{\end{equation}}

\newcommand{\A}{{\mathbb{A}}}
\newcommand{\C}{{\mathbb{C}}}
\newcommand{\DD}{\Delta}

\newcommand{\ED}{E_{\Dh}}

\newcommand{\MUD}{MU_{\Dh}}

\newcommand{\MUhs}{MU_{\hs}}

\newcommand{\Q}{{\mathbb{Q}}}
\newcommand{\R}{{\mathbb{R}}}
\newcommand{\Z}{{\mathbb{Z}}}

\newcommand{\Ab}{\mathrm{Ab}}

\newcommand{\cl}{\mathrm{cl}}

\newcommand{\codim}{\mathrm{codim}\,}

\newcommand{\geo}{\mathrm{geo}}

\newcommand{\Gr}{\mathrm{Gr}}
\newcommand{\Grml}{\Gr_m(\C^{m+l})}

\newcommand{\hs}{\mathrm{hs}}

\newcommand{\id}{\mathrm{id}}

\newcommand{\inc}{\mathrm{inc}}

\newcommand{\pinch}{\mathrm{pinch}}

\renewcommand{\mod}{\;\mathrm{mod}\;}

\newcommand{\pt}{\mathrm{pt}}

\newcommand{\sm}{\mathrm{sm}}
\newcommand{\simp}{\mathrm{simp}}
\newcommand{\sing}{\mathrm{sing}}
\newcommand{\Sing}{\mathrm{Sing}\,}

\newcommand{\colim}{\operatorname*{colim}}

\newcommand{\Hom}{\mathrm{Hom}}
\newcommand{\Imm}{\mathrm{Im}\,}

\newcommand{\Map}{\mathrm{Map}}
\newcommand{\Mapp}{\mathrm{Map}_{\ast}}
\newcommand{\mappsm}{\mathrm{Map}_{\ast}^{\sm}}
\newcommand{\mapsm}{\mathrm{Map}^{\sm}}
\newcommand{\Mappsm}{\mathrm{Map}_{\ast}^{\sm}}
\newcommand{\Mapsm}{\mathrm{Map}^{\sm}}

\newcommand{\Man}{\mathbf{Man}}

\newcommand{\Manc}{\mathbf{Man}_{\C}}

\newcommand{\sPre}{\mathbf{sPre}}

\newcommand{\sPrep}{\mathbf{sPre}_*}

\newcommand{\hosPre}{\mathrm{ho}\sPre}

\newcommand{\hoSp}{\mathrm{ho}\mathrm{Sp}}

\newcommand{\Ahs}{\Ah_{\hs}}

\newcommand{\hoSpPre}{\mathrm{hoSp}(\sPrep)}

\newcommand{\Thom}{\mathrm{Th}}

\newcommand{\Ah}{{\mathcal A}}

\newcommand{\Ch}{{\mathcal C}}
\newcommand{\Dh}{{\mathcal D}}

\newcommand{\Ds}{\mathscr{D}}

\newcommand{\HXAFp}{H^{n-1}\left(X;\frac{\Ah^*}{F^p}(\Vh_*)\right)}

\newcommand{\xto}{\xrightarrow}

\newcommand{\gb}{g_{\bullet}}

\newcommand{\hb}{h_{\bullet}}

\newcommand{\omb}{\omega_{\bullet}}

\newcommand{\gp}{{g_{\pitchfork}}}

\newcommand{\gsm}{g_{\sm}}

\newcommand{\rhog}{\rho_{\nabla}} 

\newcommand{\Mappf}{\Mapp^{\pitchfork}} 

\newcommand{\Mapf}{\Map^{\pitchfork}}

\newcommand{\Vh}{{\mathcal V}}

\newcommand{\into}{\hookrightarrow}

\newcommand{\hfcycle}{\gamma}
\newcommand{\geocycle}{\widetilde f}
\newcommand{\geocycles}{\widetilde{ZMU}}
\newcommand{\geocob}{\widetilde b}

\newcommand{\trivbr}{\underline{\mathbb R}}
\newcommand{\trivbc}{\underline{\mathbb C}}
\newcommand{\pbalpha}{F_{\alpha}}

\newcommand{\ManF}{\mathbf{Man}_F}

\DeclareMathOperator*{\supp}{supp}

\DeclareMathOperator*{\Graph}{Graph}

\newtheorem{theorem}{Theorem}[section]
\newtheorem{lemma}[theorem]{Lemma}
\newtheorem{prop}[theorem]{Proposition}

\theoremstyle{definition}
\newtheorem{defn}[theorem]{Definition}

\newtheorem{example}[theorem]{Example}

\newtheorem{remark}[theorem]{Remark}

\begin{document}

\author{Knut Bjarte Haus}
\address{Department of Mathematical Sciences, NTNU, NO-7491 Trondheim, Norway}
\email{knut.b.haus@gmail.com}

\author{Gereon Quick}
\thanks{The second-named author was partially supported by the RCN Project No.\,313472 {\it Equations in Motivic Homotopy}.
} 
\address{Department of Mathematical Sciences, NTNU, NO-7491 Trondheim, Norway}
\email{gereon.quick@ntnu.no}

\title{Geometric Hodge filtered complex cobordism}

\date{}

\begin{abstract}
We construct a geometric cycle model for a Hodge filtered extension of complex cobordism for every smooth manifold with a filtration on its de Rham complex with complex coefficients. 
Using a refinement of the Pontryagin--Thom construction, we construct an explicit isomorphism between our geometric model and the abstract model of Hodge filtered complex cobordism of Hopkins--Quick for every complex manifold with the Hodge filtration. 
\end{abstract}

\maketitle

\tableofcontents

\section{Introduction}
\label{sec:intro}

Characteristic classes play a fundamental role in the analysis of vector bundles. 
Let $M$ be a smooth manifold and $\pi \colon E\to M$ be a smooth vector bundle.  
While ordinary characteristic classes just depend on the topology of $M$ and $E$, the presence of a connection $\nabla$ on $E$ allows to define more refined secondary classes which depend on the geometry encoded by $\nabla$. 
One may give a unified description of primary and this type of secondary classes using the 
differential characters of Cheeger--Simons \cite{CheegerSimons}, or smooth Deligne cohomology. These are examples of differential cohomological invariants which  refine singular cohomology. 
In \cite{hs}, Hopkins--Singer show that every topological cohomology theory has a differential refinement for smooth manifolds.  
Differential refinements play, in particular, an important role in mathematical physics as they allow for a conceptual interpretation of quantization conditions in field theories (see e.g.\,\cite{CheegerSimons}, \cite{Freed}, \cite{Grady-Sati}, \cite{hs}, \cite{satischreiber}, and \cite{Witten}). 
For $K$-theory on smooth manifolds there are various geometric models for differential refinements, for example  
structured vector bundles of Simons--Sullivan \cite{SS} and differential K-theory of Freed--Lott \cite{FreedLott}.

Since many manifolds studied in mathematical physics are equipped with a complex structure, for example Calabi--Yau manifolds in string theory and mirror symmetry, it is desirable to have refinements  of topological cohomology theories which track the complex structure. 
For a complex manifold $X$, an analog of smooth Deligne cohomology is complex analytic Deligne cohomology \cite{Hodge2} which takes the Hodge filtration on forms and thereby the complex structure on $X$ into account. 
In \cite{Esnault}, Esnault studies secondary classes of flat bundles on analytic manifolds. 
In \cite{karoubi43}, Karoubi defined multiplicative $K$-theory which we may consider as a Hodge filtered extension of complex $K$-theory in our terminology. 
In \cite{hfc}, the authors construct a Hodge filtered extension, denoted $\ED$, of every rationally even cohomology theory $E$ following the work of Hopkins--Singer in \cite{hs}.  
In particular, they define a Hodge filtered extension of complex cobordism represented by the Thom spectrum $MU$.  
This theory shares the same relationship with complex analytic Deligne cohomology as differential cobordism of \cite{hs} does with smooth Deligne cohomology. 
Recently, Benoist showed in \cite{Benoist} that Hodge filtered cobordism of \cite{hfc} can be used to study subtle phenomena for non-algebraic cohomology classes for algebraic varieties. \\

The definition of the groups $\ED^n(p)(X)$, for integers $n$, $p$ and complex manifold $X$, in \cite{hfc} has the advantage that many important properties follow directly from the construction of $\ED(p)$ as a homotopy pullback as in \eqref{eq:MUDp_htpy_pullback_intro} below. 
However, it also has the disadvantage that it is rather abstract which makes it difficult to describe elements in  $\ED^n(p)(X)$ explicitly. 
It is therefore highly desirable to have an alternative and more concrete construction of $\ED^n(p)(X)$. 
To provide such a concrete and geometric description for $E=MU$ is the purpose of this paper. 
We will do this by showing that, for every integer $p$, $\MUD^*(p)$ is isomorphic to a new theory which we denote by $MU^*(p)$, without the subscript $\Dh$, whose elements are given by concrete geometric cycles and relations. 
In \cite{hausquick} we show that this geometric description allows us to define pushforward morphisms in $MU^*(p)$ for every proper holomorphic map between complex manifolds. 
This is a vast extension of the result in \cite{hfc} and is important for the study of Abel--Jacobi type invariants as in \cite{aj}. \\

We now briefly recall the construction of \cite{hfc} and will then outline the new contributions of the present paper. 
Let $\Manc$ denote the category of complex manifolds and holomorphic maps. 
The Grothendieck topology defined by open covers turns $\Manc$ into an essentially small site with enough points. 
For a topological space $Z$, let $\sing(Z)$ denotes its singular simplicial set. 
Let $E$ be a topological rationally even spectrum. 
Let $\sing(E)$ denote the spectrum of simplicial sets whose $n$th simplicial set is given by $\sing(E_n)$. 
For $\Vh_* = E_*\otimes_{\Z} \C$, let $H(\Vh_*)$ denote the spectrum whose $n$th simplicial set is the simplicial Eilenberg--MacLane space $K(\Vh_*,n)$. 
Let $\sing(E) \to H(\Vh_*)$ be a map of spectra which induces the genus $\phi_* \colon E_* \to \Vh_*$ which in degree $2k$ multiplies with $(2\pi i)^k$. 
We consider $\phi$ as a map of constant presheaves of spectra on $\Manc$. 
The canonical inclusion map $\Vh_* \to \Ah^*(X;\Vh_*)$ from constant functions into the complex of smooth forms  induces a map of presheaves of Eilenberg--MacLane spectra 
$H(\Vh_*) \to H(\Ah^*(\Vh_*))$. 
Composition then yields a map 
$\phi \colon \sing(E) \to H(\Ah^*(\Vh_*))$ 
which we also denote by $\phi$. 
For a given integer $p$, let $\phi^p=(2\pi i)^p\cdot \phi$. 
Then $\ED(p)$ is defined by the homotopy cartesian square of presheaves of spectra on $\Manc$ 
\begin{align}\label{eq:MUDp_htpy_pullback_intro}
\xymatrix{
\ED(p)\ar[r]\ar[d] & \sing(E) \ar[d]^{\phi^p}\\
H(F^p\Ah^*(\Vh_*))\ar[r] & H(\Ah^{*}(\Vh_*))
}
\end{align}
where the notation for the presheaves of the complex of filtered forms $F^p\Ah^*(\Vh_*)$ is explained in \eqref{eq:Hodge_filtration} and Definition \ref{def:extend_filtration}. 
Let $\hoSp(\sPre_*)$ denote the homotopy category of presheaves of spectra on $\Manc$. 
Then the $n$-th Hodge filtered $E$-cohomology group with  twist $p$ of a complex manifold $X$ is defined as the group of homotopy classes of maps of presheaves of spectra 
\begin{align*}
\ED^n(p)(X) = \Hom_{\hoSp(\sPre_*)}(\Sigma^{\infty}(X_+), \Sigma^n\ED(p)). 
\end{align*}

We will now summarise the construction of our geometric model $MU^*(p)$ which is inspired by Karoubi's multiplicative K-theory of \cite{karoubi43} and the geometric differential complex cobordism groups of \cite{Bunke2009}.  
The details are given in section \ref{Section:GeometricHFCBordism}.  
For a smooth manifold $X$ we denote by $\Ah^*$ the sheaf on $X$ given by the de Rham complex with complex coefficients. 
Similar to \cite{karoubi43} we develop our theory for objects in the category $\ManF$ which are pairs $(X,F^*)$ where $X$ is a smooth manifold, and $F^*$ a descending filtration of  $\Ah^*$ on $X$ as a chain complex of $\Ah^0$-modules. 
A morphism $f\colon X\to Y$ in $\ManF$ is a smooth map $f\colon X\to Y$ such that, for each $p$, we have $f^*F^p\Ah^*\subset F^p\Ah^*$. 
Whenever $X$ is a complex manifold we consider it as an object in $\ManF$ together with the Hodge filtration on $\Ah^*$. 
Since we are mostly interested in the case of complex manifolds we will refer to $MU^*(p)(X)$ which we will define below as geometric \emph{Hodge filtered} complex cobordism even though $X$ may just be in $\ManF$. 
Our main reason to use the category $\ManF$ is that it makes it easier to work with products of a complex manifold and a smooth manifold. 
We hope, however, that the additional generality may turn out to be useful in future applications as well.  \\

Let $(X,F^*)$ be an object in $\ManF$. 
Recall from \cite{karoubi43} that Karoubi's multiplicative K-theory groups $MK(X)$ are generated by triples $(E,\nabla, \omega)$ where $E\to X$ is a complex vector bundle with connection $\nabla$, and $\omega$ is a sequence of forms such that $ch_{2p}(\nabla)+d\omega_{2p}\in F^p\Ah^{2p}(X)$. 
Here $ch_{2p}(\nabla)$ denotes the $2p$-th Chern--Weil Chern character form of $\nabla$.     
For Hodge filtered complex cobordism, we essentially replace vector bundles with connection by the differential cobordism cycles of \cite{Bunke2009}.

Consider the genus $\phi \colon MU_*\to \Vh_*$ given by multiplication by $(2\pi i)^{n}$ in degree $2n$. By Thom's theorem, $MU_n$ is the bordism group of $n$-dimensional almost complex manifolds. Hirzebruch showed that if $R$ is an integral domain over $\Q$, then any genus $\phi\colon MU_*\to R$ is of the form 
$$
\phi(Z)=\int_Z(K^\phi(TZ))^{-1}
$$
for a multiplicative sequence $K^\phi$, which yield an $R$-valued characteristic class of complex vector bundles. 
Now we set $R=\Vh_* = MU_*\otimes \C$ and  consider the characteristic class $K^p=(2\pi i)^p\cdot K^\phi$. 
If $\nabla$ is a connection on a complex vector bundle $E\to X$, Chern--Weil theory gives a form $K^p(\nabla)$ representing $K^p(E)$. 
Given a proper oriented map $f\colon Z\to X$ and a form $\omega$ on $Z$, we consider the pushforward current $f_*\omega$, which acts 
by $\sigma\mapsto \int_Z\omega\wedge f^*\sigma$.

Now we can describe the group of Hodge filtered cycles $ZMU(p)(X)$. The details are explained in section \ref{Section:GeometricHFCBordism}. 
The elements in $ZMU(p)(X)$ are triples $(f,\nabla,h)$ where $f$ is a proper, complex oriented map $f\colon Z\to X$, $\nabla$ is a connection on the complex stable normal bundle of $f$ and $h$ is a current on $X$ so that $f_*K^p(\nabla)- dh$ is a smooth form in $F^p\Ah^*(X;\Vh_*)$.  
We grade $ZMU(p)(X)$ by the codimension of $f$, so for $(f,\nabla,h)\in ZMU^n(p)(X)$ we have $\dim X-\dim Z=n$.
Whenever $h$ is a form with $dh\in F^p\Ah^n(X;\Vh_*)$, we define $a(h)=(0,0,h)\in ZMU^n(p)(X)$. 
We will quotient out the group of cycles $a\left(F^{p}\Ah^{n-1}\right)$.
Now we define the cobordism relation, which is essentially that of \cite{Bunke2009}.
Let $\widetilde{b}=(b,\nabla)$ be a pair with $b \colon W\to \R\times X$ and $\nabla$ a connection on its stable normal bundle. 
We assume that $b$ is transverse to the inclusion $\iota_t\colon X\to \R\times X$, given by $\iota_t(x)=(t,x)$, for $t=0,1$. 
Let 
$$
\psi(\widetilde{b}) = \int_{[0,1]\times X/X} b_*K^p(\nabla)
$$
and define $BMU_{\geo}^n(p)(X)\subset ZMU^n(p)(X)$ as the subgroup generated by cycles of the form
\[
(f_1,\nabla_1,0)-(f_0,\nabla_0,\psi(\widetilde{b})).
\]

\begin{defn}
For $(X,F^*) \in \ManF$ and integers $n,p$, the geometric Hodge filtered complex cobordism groups are defined by
$$  
MU^n(p)(X) := \frac{ZMU^n(p)(X)}{BMU_{\geo}^n(p)+a\left(F^{p}\Ah^{n-1}(X;\Vh_*)\right)}.
$$
\end{defn}

We now state the main result of this paper: 

\begin{theorem}\label{thm:mainTheoremKappaintro}
Let $X$ be a complex manifold together with the Hodge filtration. 
Then, for all integers $n,p$, there is an isomorphism of Hodge filtered cohomology groups
\begin{align}\label{eq:mainthmiso}
\MUD^n(p)(X) \cong MU^n(p)(X)
\end{align}
which respects pullbacks along holomorphic maps.  
\end{theorem}


\begin{remark}
The analog of Theorem \ref{thm:mainTheoremKappaintro} for differential cobordism is a consequence of the uniqueness theorem of \cite{Bunke2010} for differential extensions of cohomology theories which satisfy the axioms of \cite{Bunke2010}.   
We propose similar axioms for Hodge filtered extensions in section \ref{sec:axioms}.  
However, the proof of the uniqueness theorem of \cite{Bunke2010} relies on the fact that any continuous map is homotopic to a smooth one. 
The corresponding assertion for holomorphic maps is false. 
We have not succeeded in finding a proof that works also in the holomorphic setting. 
We therefore prove Theorem \ref{thm:mainTheoremKappaintro} by providing an explicit isomorphism.
\end{remark}

The proof of theorem \ref{thm:mainTheoremKappaintro} 
is based on the extent to which the Pontryagin--Thom construction is compatible not just with topological but differential geometric data as well. To check all the required compatibilities requires a detailed analysis of the geometry of the Pontryagin--Thom construction. 
We then construct isomorphism \eqref{eq:mainthmiso} as a composition of two isomorphisms. 
In section \ref{sec:homotopical_model} we construct a more geometric model of Hodge filtered cobordism on $\ManF$. 
We prove in section \ref{sectionComparisonOfHFCTheories} that this model and the one of \cite{hfc} are equivalent. 
Then we define in section \ref{section:from_htpy_to_geom} a map from the new model to the geometric cycle description we described above. 
While at first glance it may seem to be a rather straight forward task to construct a cycle out of the data of homotopy pullback \eqref{eq:MUDp_htpy_pullback_intro} for $E=MU$, we do not expect it to be possible to construct isomorphism \eqref{eq:mainthmiso} in a more direct way.

We will now further describe the key ideas for the construction of isomorphism \eqref{eq:mainthmiso}. 
Consider a pointed continuous map 
\begin{align}\label{eq:intro_gmap_smooth_transverse}
\xymatrix{\Thom(\R^k\times X)=\Sigma^kX_+\ar[r]^-g& MU(m,l)=\Thom(\gamma_{m,l})}
\end{align}
such that $g|_{g^{-1}(\gamma_{m,l})}$ is smooth and transverse to $\Grml$, considered as the image of the $0$-section $\iota_{m,l}$ in $\gamma_{m,l}$. Here $\Thom(E)$ is the Thom space of $E$.  
Then we get a manifold $Z_g=g^{-1}(\Grml)$, and a natural map $f_g\colon Z_g\to X$. 
The map $f_g$ is naturally complex oriented, and the associated complex representative of the stable normal bundle is $N_g=\left(g|_{Z_g}\right)^*\gamma_{m,l}$. 
There are natural compatible connections $\nabla_{m,l}$ on $\gamma_{m,l}$, as considered in \cite{NR1}, 
which we may pull back to get a connection $\nabla_g=\left(g|_{Z_g}\right)^*\nabla_{m,l}$. 
Now we view the currents $\phi_{\nabla_{m,l}}:=\left(\iota_{m,l}\right)_*K^p(\nabla_{m,l})$ as currents on $MU(m,l)$. 
As such, they should correspond to maps of presheaves $MU(m,l)\to H(\Ds^*(\Vh_*))$, where we consider $MU(m,l)$ as representing the presheaf of smooth maps to $MU(m,l)$ and use $\Ds^*$ to denote the complex of currents. 
Then we may replace the map $\phi^p \colon \sing (MU) \to H(\Ah^*(\Vh_*))$ in \eqref{eq:MUDp_htpy_pullback_intro} with a map $\phi_{\nabla}\colon MU \to H(\Ds^*(\Vh_*))$ defined by the formula $\phi_\nabla(g)=\pi_*g^*(\phi_{\nabla_{m,l}})$ where $\pi\colon \R^k\times X\to X$ is the projection. 
This would harmonise well with the geometric groups $MU^n(p)(X)$, since we have 
\begin{equation} \label{fundamentalCurrentEquation}
\left(f_g\right)_*K^p(\nabla_g)=\pi_*\left(g|_{g^{-1}(\gamma_{m,l})}\right)^*\phi_{\nabla_{m,l}}.
\end{equation}
Hence, if we furthermore had a current $h$ such that $\omega=\phi_{\nabla}(g)- dh$ is in $F^p\Ah^n(X;\Vh_*)$, we could simply map $(g,h,\omega)$ to $(f_g,\nabla_g,h)$. 
However, since currents may not be pulled back along arbitrary holomorphic maps, $\Ds^*$ is not a presheaf and $\phi_{\nabla}$ cannot be made into a map of presheaves. 

Therefore, we note that $\left(\iota_{m,l}\right)_*$ induces the Thom isomorphism, and so we may use compatible Thom forms to circumvent the use of currents in the above analysis. 
We chose to use the Mathai--Quillen Thom forms of \cite{mathai-quillen}, which are natural for Hermitian bundles with unitary connections.
Then we can pull back these Thom forms along maps $g$ as in \eqref{eq:intro_gmap_smooth_transverse}. 
We denote the set of such maps by $\Mappsm(\Sigma^k X_+,MU(m,l))$. 
Now we would like to take colimits and to replace the spaces $MU(m,l)$ with the spaces $MU_n$, where $n+k=2m$, and then with   $QMU_n=\colim_{k}\Omega^kMU_{n+k}$. 
We show that the Mathai--Quillen Thom forms behave well with respect to these two stabilisation procedures. 
Finally, we consider the simplicial set  $\mapsm(\DD^{\bullet} \times X, QMU_n)$ of, in an appropriate sense, smooth maps. 
Note that we work with $\DD^{\bullet} \times X$ instead of $X \times \DD^{\bullet}$, since the orientation of the former is compatible with integration over $\DD^{\bullet}$. We refer to 
Remark \ref{rem:order_of_Delta_and_X_integration} for further details. 

In section \ref{sec:new_model_MUhs(p)} we then define a new homotopy theoretic model, denoted $\MUhs(p)$, of Hodge filtered cobordism using the existence of the Mathai--Quillen forms. 
A key observation is that the latter will allow us to define natural maps of simplicial sets 
\begin{align*}
\phi_{\sm}^{n} \colon \mapsm(\Delta^{\bullet} \times X, QMU_n) \to \Ah^n(\Delta^{\bullet} \times X;\Vh_*). 
\end{align*}
The classes in $\MUhs^n(p)(X)$ are represented by triples $(g,\omega,h)$ where $g\colon \Sigma^kX_+\to MU(m,l)$ is smooth, $\omega\in F^p\Ah^n(X;\Vh_*)$, and $h\in \Ah^n(\Delta^1 \times X;\Vh_*)$ restricts to $\omega$ at one end and $\phi^n_{\sm}(g)$ at the other. See Lemma \ref{thm:HFC_concrete_triples} for a precise statement.

We show in section \ref{sec:MUhs_is_a_presheaf_of_spectra} that the new model $\MUhs(p)$ is represented by a presheaf of spectra which fits into a homotopy pullback similar to the one defining $\MUD(p)$. 
The constant presheaf $\sing(MU)$ is replaced with the presheaf of simplicial spectra $X\mapsto \mapsm(\DD^{\bullet} \times X, QMU)$. 
The map $\phi_{\sm}^n$ above induces a map of presheaves of spectra 
$$
\phi_{\sm}\colon \Mapsm(\DD^{\bullet} \times -, QMU)\to \Ah^*_{\hs}(\Vh_*),
$$
where $\Ah^*_{\hs}(\Vh_*)$ is a presheaf of simplicial spectra  over $\Manc$ which is weakly equivalent to the Eilenberg--Maclane spectrum $H(\Ah^*(\Vh_*))$. 
We also define a spectrum $F^p\Ahs^*(\Vh_*)$ which is weakly equivalent to $H(F^p\Ah^*(\Vh_*))$. 
They are related by a natural map  $F^p\Ahs^*(\Vh_*)) \to \Ahs^*(\Vh_*)$ which is induced by the objectwise inclusion of sheaves $F^p\Ah^*(\Vh_*) \xto{\inc} \Ah^*(\Vh_*)$. 
The presheaf of spectra $\MUhs(p)$ is then defined as the homotopy pullback of the diagram 
\[
\Mapsm(\DD^{\bullet} \times -, QMU)\to \Ah^*_{\hs}(\Vh_*) \leftarrow F^p\Ah^*_{\hs}(\Vh_*).
\]
Our definition of the spectra $\Ah^*_{\hs}(\Vh_*)$ and $F^p\Ah^*_{\hs}(\Vh_*)$ follows essentially the work of Hopkins--Singer \cite[Appendix D]{hs}, which is the motivation for the subscript $\hs$. 
In Theorems \ref{thm:htpy_pullbacks_MUD_MUhs_are_we} and \ref{thm:map_MUD_to_MUhs_is_an_iso} 
in section \ref{subsec:comparison_of_htpy_models} we show that there is a weak equivalence of presheaves of spectra $\MUD(p) \simeq \MUhs(p)$ and hence a natural  isomorphism $\MUD^*(p)(X) \cong \MUhs^*(p)(X)$ for every $p \in \Z$.  

Since $\MUhs(p)$ is a much more accessible model, we are then able to define in section \ref{subsec:construction_of_kappa} a natural map
\[
\kappa \colon \MUhs^n(p)(X)\to MU^n(p)(X)
\]
by setting 
$$
\kappa(g,\omega,h) = \left(f_g,\nabla_g,\pbalpha(g) + \int_{\Delta^1 \times X/X} h \right).
$$
Here $\pbalpha(g)$ is a universally defined correction term which is needed since we used a Thom form on $\gamma_{m,l}$ instead of the Thom current $\left(\iota_{m,l}\right)_*K(\nabla_{m,l})$ appearing in \eqref{fundamentalCurrentEquation}. 
The assignment $X \mapsto \MUhs^*(p)(X)$ is actually defined for every $X\in \ManF$. 
In fact, we show in Theorem \ref{thm:kappa_is_an_isomorphism} in section \ref{subsec:kappa_is_an_iso} that $\kappa$ is an isomorphism of Hodge filtered extensions over $\ManF$.  \\

Many of the ideas of the present paper appeared in the doctoral thesis of the first-named author. 
Both authors would like to thank the Department of Mathematical Sciences at NTNU for the continuous support during the work on the thesis and this paper. 
We thank Mike Hopkins for helpful discussions and the anonymous referees for helpful comments and suggestions to improve the exposition of the paper.


\section{Geometric Hodge filtered cobordism} \label{Section:GeometricHFCBordism}

In this section we will define geometric Hodge filtered complex cobordism groups. 
We begin with some recollection and notation.

\subsection{Currents}

Let $X$  be a smooth manifold and let $\Lambda_X$ denote the orientation bundle of $X$. 
Let $\Ah_c^*(X;\Lambda_X)$ be the space of compactly supported smooth forms on $X$ with values in $\Lambda_X$.
Let $\Ds^*(X)$ denote the space of currents on $X$, defined as the topological dual of $\Ah_c^*(X;\Lambda_X)$. 
Given a form $\omega\in\Ah^*(X)$ and a current $T\in \Ds^*(X)$, their product acts by
$$
T\wedge \omega (\sigma)= T(\omega\wedge \sigma).
$$ 
There is an injection $\Ah^*(X)\to \Ds^*(X)$ given by
$$
\omega \mapsto T_\omega =\left(\sigma\mapsto \int_X\omega\wedge \sigma ,\quad \sigma\in \Ah^*_c(X;\Lambda_X)\right).
$$
We grade $\Ds^*$ so that this injection preserves degree. 
That is, $\Ds^k(X)$ consists of the currents which vanish on a homogeneous $\Lambda_X$ valued form $\sigma$, unless possibly if $\deg \sigma = \dim_\R X-k$. 
We will not always distinguish $\omega$ from $T_\omega$ in our notation.

If $X$ is a manifold without boundary, Stokes' theorem implies for $\omega\in \Ah^k(X)$: 
$$
T_{d\omega}(\sigma)=(-1)^{k+1}T_\omega(d\sigma).
$$
Hence we can extend the exterior differential to a map $d\colon \Ds^k(X)\to \Ds^{k+1}(X)$ by 
$$
dT(\sigma)=(-1)^{k+1}T(d\sigma). 
$$
We define for a vector space $V$, $\Ds^*(X;V)=\Ds^*(X)\otimes V$, and for a graded vector space $\Vh_*$ we set 
$$
\Ds^n(X;\Vh_*)=\bigoplus_j \Ds^{n+2j}(X;\Vh_{2j}).
$$
An orientation of a map $f\colon Z\to X$ is equivalent to an isomorphism $\Lambda_Z\simeq f^*\Lambda_X$. 
If $f$ is proper and oriented, we therefore get a map 
\[
f^*\colon \Ah_c^*(X;\Lambda_X)\to \Ah_c^*(Z;\Lambda_Z)
\] 
which induces a map 
\[
f_*\colon\Ds^*(Z)\to \Ds^{*+ k}(X)
\]
where $k=\codim f=\dim X-\dim Z$.
We also denote by $f_*$ the homomorphism $\Ds^*(Z;\Vh_*)\to \Ds^*(X;\Vh_*)$ induced by tensoring $f_*$ with the identity of the various $\Vh_{2j}$. 
We get the equality 
$$
d\circ f_*=(-1)^kf_*\circ d.
$$

\begin{remark}\label{IntegrationOverTheFiber}
In the case of a submersion $\pi\colon W\to X$ the pushforward $\pi_*$ takes forms to forms. 
We thus obtain the \emph{integration over the fiber} map
$$
\int_{W/X}\colon \Ah^*(W)\to \Ah^{*+d}(X)
$$ 
defined by the equation
$$
T_{\int_{W/X}\omega} = \pi_*T_\omega.
$$
\end{remark}


\subsection{The category \texorpdfstring{$\ManF$}{ManF}}

We recall the definition of $\ManF$ from the introduction. 

\begin{defn}
\label{def:ManF}
We denote by $\ManF$ the category with objects pairs $(X,F^*)$ where $X$ is a smooth manifold, and $F^*$ a descending filtration of the sheaf $\Ah^*$ on $X$ given by the de Rham complex with complex coefficients as a chain complex of $\Ah^0$-modules. 
We will often just write $X\in\ManF$, in which case $F^p\Ah^*$ will refer to the filtration associated to $X$ as an object of $\ManF$. 
A morphism $f\colon X\to Y$ in $\ManF$ is a smooth map $f\colon X\to Y$ such that, for each $p$, we have $f^*F^p\Ah^*\subset F^p\Ah^*$.
\end{defn}

\begin{remark}
Let $\Man$ denote the category of smooth manifolds. 
There is an embedding of categories $\Man\to \ManF$ which endows a manifold with the trivial filtration $\Ah^*=F^0\Ah^*\supset F^1\Ah^*=0$. 
Whenever we consider a smooth manifold $S$ without specifying a filtration we equip $S$ with this trivial filtration. 
\end{remark}

\begin{remark}
There is an embedding of categories $\Manc\to \ManF$ given by considering a complex manifold $X$ together with the Hodge filtration on $\Ah^*$, i.e.,  
whenever $X$ is a complex manifold we consider it as an object in $\ManF$ with the \emph{Hodge filtration on forms} which is defined by 
\begin{align}\label{eq:Hodge_filtration}
F^p\Ah^n(X) = \bigoplus_{i\ge 0}\Ah^{p+i, n-p-i}(X).
\end{align}
This filtration is not associated with a Hodge structure, but the name is justified by the fact that if $X$ is compact K\"ahler, \eqref{eq:Hodge_filtration} induces the Hodge filtration on $H^*(X;\C)$. 
In fact, we are mainly interested in the case of complex manifolds. 
We develop our theory on the category $\ManF$ instead of $\Manc$ because it is convenient to have a theory that naturally handles products $S \times X$ for $S$ merely smooth.
\end{remark}

\begin{defn}
\label{def:Hodge_filtration_on_product}
Given $X_1, X_2\in \ManF$, we equip $X_1\times X_2$ with the filtration
$$
F^p\Ah^*(X_1\times X_2)=\bigoplus_{p_1+p_2=p} \overline{F^{p_1}\Ah^*(X_1)\otimes F^{p_2}\Ah^*(X_2)}
$$
where the over-line denotes the closure in the topological space $\Ah^*(X\times Y)$ equipped with  the $\Ch^\infty$-compact-open topology of \cite{deRham}. 
Thus a form $\omega= f\cdot \omega_X\otimes \omega_Y$, for $f$ a function on $X\times Y$, belongs to $F^{p_1+p_2}\Ah^*(X\times Y)$ whenever $\omega_X\in F^{p_1}\Ah^*(X)$ and $\omega_Y\in F^{p_2}\Ah^*(Y)$, and every element of $\Ah^*(X\times Y)$ is a finite sum of such forms.
\end{defn}

This construction is not the categorical product on $\ManF$. 
It is, however, the categorical product on the full subcategory of $\ManF$ satisfying the following condition:
\begin{equation}\label{productCondition}
    \text{For }\omega_i\in F^{p_i}\Ah^*(X),\ i=1,2,\text{ we have }\omega_1\wedge \omega_2\in F^{p_1+p_2}\Ah^*(X).
\end{equation}

\begin{remark}
Note that condition \eqref{productCondition} is equivalent to requiring $\Delta\colon X \to X\times X$ to be a morphism in $\ManF$. 
\end{remark}
 
\begin{remark}
\label{remark:Hodge_filtration_on_product}
Let $X$ be a complex manifold 
and let $S$ be a smooth manifold. 
Then we can describe the filtration for the product  $S \times X$ as follows. 
Let $(z_i)$ be holomorphic coordinates for $U\subset X$, and let $(s_i)$ be smooth coordinates for $V\subset S$. 
Then a form
$$
\sum_{IJK}f_{IJK} ds_K \wedge dz_I\wedge d\overline{z}_J \in \Ah^*(V\times U)
$$
belongs to $F^p\Ah^*(V\times U)$ exactly if $f_{IJK}=0$ whenever $|I|<p$. 
By slight abuse of terminology, we will often refer to this filtration as the Hodge filtration on $S \times X$. 
\end{remark}

\begin{defn}\label{def:extend_filtration}
Let $\Vh_*$ be an evenly graded $\C$-vector space and $X\in \ManF$. 
We extend the given filtration on $\Ah^*(X)$ to forms with coefficients in $\Vh_*$ by
\begin{align*}
F^p\Ah^n(X;\Vh_*) :=\bigoplus_{j\in \Z} F^{p+j}\Ah^{n+2j}(X;\Vh_{2j}).
\end{align*}
The grading is defined such that a form $\omega \in \Ah^r(X;\Vh_s)$ has degree $r-s$.
\end{defn}


\subsection{Cycle model for \texorpdfstring{$MU(X)$}{MUX}} 

We now recall from \cite{quillen} Quillen's description of the complex cobordism groups of a smooth manifold $X$, denoted by $MU^n(X)$.
Let $f\colon Z \to X$ be a smooth map. We may factorize $f$ as $f=\pi\circ\iota$ for an embedding $\iota\colon Z\to X \times \C^N$ and $\pi$ the projection onto $X$. 
If the codimension of $f$, defined by $\codim f=\dim X-\dim Z$, is even, then a complex orientation of $f$ is represented by a complex structure on the normal bundle of $\iota$. 
If $f$ has odd codimension, a complex orientation of $f$ is a complex orientation of the map $Z\to X \times \R$ given by $z\mapsto (f(z),0)$. 
In either case we obtain a complex vector bundle $N_f$ which represents the stable normal bundle of $f$. 
Two choices of factorization and complex structures, the second one denoted by primes, represent the same complex orientation if there is a commutative diagram
$$
\xymatrix{
Z\ar[r]^-\iota\ar[d]_-{i_0}& X \times \C^N\ar[d]\ar[dr]^-\pi \\
Z\times I \ar[r] & X \times \C^{N''}\ar[r] & X\\
Z\ar[r]_-{\iota'}\ar[u]^-{i_1}& X \times \C^{N'}\ar[ur]_-{\pi'}\ar[u]
}
$$
where the central vertical maps are linear embeddings of complex vector bundles,
and the first central horizontal map is an isotopy between the two maps $Z\to X \times \C^{N''}$ through factorizations of $f$. 
Thus the class of the stable complex normal bundle $[N_f]=[f^*TX]-[TZ]\in K^0(Z)$ depends only on the complex orientation.
If $f_1$ and $f_2$ are composable complex oriented maps, their composition is complex oriented with complex normal bundle satisfying $[N_{f_2\circ f_1}] = [N_{f_1}] + f_1^*[N_{f_2}]$.
Two complex oriented maps $f_1\colon Z_1\to X$ and $f_2\colon Z_2\to X$ are isomorphic if there is a diffeomorphism $\psi\colon Z_1\to Z_2$ so that $f_1=f_2\circ \psi$ as an equality of complex oriented maps.

A complex cobordism cycle, or cycle for short, is an isomorphism class of proper complex oriented maps $f\colon Z\to X$. We let $ZMU^n(X)$ denote the monoid of cycles of codimension $n$ under disjoint union.
A cycle 
$$
b=(c,f)\colon W\to \R\times X\in ZMU^n(\R\times X)
$$ 
is called a bordism datum over $X$ if $0$ and $1$ are regular values of $a\colon W\to \R$. 
In this case we define $W_t=c^{-1}(t)$, and $f_t=f|_{W_t}$ for $t=0,1$, and put 
$$
\partial b = f_1+(-f_0)
$$
where $-f_0$ denotes the cycle obtained from $f_0$ by reversing the complex orientation of its representing complex oriented maps. Then we define $BMU^n(X)$ as the submonoid generated by boundaries $\partial b$ as $b$ range over all bordism data over $X$.
The $n$-th complex cobordism group of $X$ is then defined by 
$$
MU^n(X)=ZMU^n(X)/BMU^n(X).
$$
In fact, 
$MU^n(X)$ is contravariantly functorial in $X$. 
For the pullback operation, let $g\colon Y\to X$ be a smooth map. 
If the cycle $f\colon Z\to X$ is transverse to $g$, then $g^*[f]$ is represented by $f'$ in the following pullback square
$$
\xymatrix{
Z'\ar[r]\ar[d]_{f'} & Z\ar[d]^f\\
Y\ar[r]_g & X.
}
$$
It follows from Thom's transversality theorem that each cobordism class $[f]$ can be represented by a map which is transverse to $g$. 
That the cobordism class of $g^*[f]$ depends only on $[f]$ follows by similarly pulling back bordism data. Hence $g^*\colon MU^n(X)\to MU^n(Y)$ is well-defined.
There is an exterior product 
$$
MU^n(X)\times MU^m(Y)\to MU^{n+m}(X\times Y)
$$
given by $([f],[g])\mapsto [f\times g]=:[f]\times [g]$. Then $MU^*(X)$ is turned into a ring by $[f]\cdot [g] := \Delta^*([f]\times [g])$, for $\Delta\colon X\to X\times X$ the diagonal map $\Delta(x)=(x,x)$.

\subsection{Genera}
We let $MU_*$ be the graded ring with $MU_n=MU^{-n}(\pt)$. A map of rings $MU_*\to R$ for $R$ an integral domain over $\Q$ is called a complex genus. 
We recall from \cite{HBJ} that complex genera may be constructed in the following way. 
For each $i\in \mathbb N$, let $x_i$ be a indeterminate of degree $i$. 
Let $Q\in R[[y]]$ be a formal power series in the variable $y$ of degree $2$. Let $\sigma_i$ denote the $i$-th elementary symmetric function in $x_1,x_2,\dots $. We may then define a sequence of polynomials $K_i^Q$ satisfying 
$$
K^Q(\sigma_1,\sigma_2,\cdots)=1+K^Q_2(\sigma_1)+K^Q_4(\sigma_1,\sigma_2)+\dots = \prod_{i=1}^\infty Q(x_i),
$$
since the right hand side is symmetric in the $x_i$.
Then we get a characteristic class $K^Q$, defined on a complex vector bundle $E\to X$ of dimension $n$ by
$$
K^Q(E):=K^Q(c_1(E),\dots, c_n(E))\in H^*(X;R).
$$
By \cite[section 1.8]{HBJ} all genera are of the form
$$
\phi^Q([X])=\int_XK^Q(N_X)
$$
where $N_X$ denotes the complex vector bundle representing the stable normal bundle of $X$ obtained from the complex orientation of $X\to \pt$. 
We now suppose that $\Vh_*$ is a graded ring and that the power series $Q(y)=1+r_1y+r_2y^2+\cdots $ has total degree $0$. 
This is equivalent to assuming $\phi^Q$ to be a degree-preserving genus. 
Then $K^Q(E)$ has total degree $0$. 
From now on we set $\Vh_*:=MU_*\otimes_{\Z}\C$. 
By \cite[Lemma 3.26]{Bunke2009} $\phi^Q$ extends to a morphism of multiplicative cohomology theories 
$$
\phi^Q\colon MU^n(X)\to H^n(X;\Vh_*)
$$
by 
\begin{equation}\label{DefinitionNaturalTransformationphi}
\phi^Q([f])=f_*K^Q\left(N_f\right).
\end{equation}
Here $H^n(X;\Vh_*)\cong \bigoplus_j H^{n+2j}(X;\Vh_{2j})$, so that in particular $H^{-2j}(\pt;\Vh_*)\simeq \Vh_{2j}$. 

\begin{defn}
\label{def:def_of_phi_and_phip_on_cohomology}
We fix the multiplicative natural transformation  
$$
\phi\colon MU^*(X)\to H^*(X;\Vh_*)
$$
characterized by restricting to multiplication with $(2\pi i)^k$ on $MU_{2k} \to MU_{2k}\otimes \C$. 
Let 
$$
K = 1+K_2(\sigma_1)+K_4(\sigma_1,\sigma_2)+ \cdots
$$ 
be the multiplicative sequence satisfying $\phi([f])=f_*K(N_f)$.
For $p\in \Z$ we set $K^p=(2\pi i)^p \cdot K$ and
\begin{align}
\label{eq:def_of_phip}
\phi^p([f])=f_*K^p(N_f).
\end{align}
\end{defn}

Let $f\colon Z\to X$ be complex oriented, and let $\nabla$ be a connection on $N_f$. By Chern--Weil theory, there is a well-defined form $c(\nabla)\in \Ah^*(Z)$ representing the total Chern class $c(N_f)$. In fact with respect to local coordinates we have 
$$
c(\nabla)=1+2 c_1(\nabla)+ c_2(\nabla)+\dots =\det\left(I-\frac{1}{2\pi i}F^\nabla\right)
$$
where $F^\nabla$ denotes the curvature of $\nabla$. Then
$$
K(\nabla):=K(c_1(\nabla),c_2(\nabla),\dots)\in\Ah^0(Z;\Vh_*)
$$
represents the class $K(N_f)$.


\subsection{Definition of geometric Hodge filtered cobordism groups}

First we recall the definition of geometric cobordism cycles from \cite{Bunke2009}: 

\begin{defn}
A geometric cycle over $X$ is a triple $\geocycle=(f,N,\nabla)$ where $f$ is a proper complex oriented map, with $N$ a complex vector bundle representing the stable normal bundle of $f$, and $\nabla$ a connection on $N$.
We say that $\geocycle$ and $\geocycle'$ are isomorphic if there is an isomorphism $g\colon Z\to Z'$ of complex oriented maps such that, under the induced isomorphism $N \cong g^*N'$, $\nabla$ and $g^*\nabla'$ are identified.
Let 
$$
\geocycles^n(X)
$$ 
denote the abelian group generated by isomorphism classes of geometric cycles over $X$ of codimension $n$ with the relations $\geocycle_1+\geocycle_2 = \geocycle_1\sqcup \geocycle_2$.
\end{defn}

\begin{defn}
\label{def:def_of_phip_current}
Let $K^p$ be as in Definition \ref{def:def_of_phi_and_phip_on_cohomology} for the multiplicative natural transformation $\phi$. 
For a geometric cycle $\geocycle\in \geocycles^n(X)$ we define, using the orientation of $f$ induced by its complex orientation, the current 
\begin{align}
\label{eq:def_of_phip_current}
\phi^p(\geocycle)=f_*K^p(\nabla_f)\in \Ds^n(X;\Vh_*).
\end{align}
\end{defn}

\begin{remark}
Note that $\phi^p(\geocycle)$ is a closed current representing the cohomology class $\phi^p([f])=f_*K^p(N_f)\in H^n(X;\Vh_*)$ defined in \eqref{eq:def_of_phip}. 
By de Rham's theorem \cite[Theorem 14]{deRham} we can find a current $h\in \Ds^{n-1}(X;\Vh_*)$ such that 
\begin{align}
\label{eq:we_can_find_currecnt_h}
\phi^p(\geocycle) - dh =  f_*K^p(\nabla_f) - dh ~ \text{is a form, i.e., lies in} ~ \Ah^n(X;\Vh_*).    
\end{align} 
This observation will be crucial for the definition of Hodge filtered cycles below. 
\end{remark}

\begin{defn}\label{def:ZMUn(p)(X)}
Let $(X,F^*)$ be an object in $\ManF$. 
We define the group of \emph{Hodge filtered cycles} of degree $(n,p)$ on $(X,F^*)$ as the subgroup 
$$
ZMU^{n}(p)(X)\subset \geocycles^n(X)\times \Ds^{n-1}(X;\Vh_*)/d\Ds^{n-2}(X;\Vh_*)
$$ 
consisting of pairs $\hfcycle=(\geocycle, h)$ satisfying
\begin{align*}
\phi^p(\geocycle) - dh \in F^p\Ah^n(X;\Vh_*)
\end{align*}
where $\phi^p(\geocycle)$ is defined by \eqref{eq:def_of_phip_current} in Definition \ref{def:def_of_phip_current}.  

\end{defn}

\begin{remark}\label{phipremark}
To simplify the notation, we will often write $\phi$ instead of $\phi^p$.
We may sometimes write a Hodge filtered cobordism cycle as a triple 
$$
\hfcycle = (\geocycle, \omega, h)\in \geocycles^n(X)\times F^p\Ah^n(X;\Vh_*)\times \Ds^{n-1}(X;\Vh_*)/d\Ds^{n-2}(X;\Vh_*),
$$ 
where $(\geocycle, h)\in ZMU^n(p)(X)$ and $\phi^p(\geocycle)- dh=\omega$.
\end{remark}

We now define maps on the level of cycles as follows:
\begin{alignat}{2}\label{cyclstructmaps}
    \nonumber	& R \colon ZMU^n(p)(X)\to F^p\Ah^n(X;\Vh_*)_{\cl},   & & R(\geocycle ,h) = \phi^p(\geocycle)- dh\\
 & a \colon d^{-1}\left(F^p\Ah^n(X;\Vh_*)\right)^{n-1} \rightarrow ZMU^n(p)(X),\quad \quad  & & a(h) = (0,h)\\
\nonumber & I \colon ZMU^n(p)(X)\rightarrow ZMU^n(X),    &&I(\geocycle ,h)= f
\end{alignat} 
where $d^{-1}\left(F^p\Ah^n(X;\Vh_*)\right)^{n-1}$ denotes the subset of elements in $\Ah^{n-1}(X;\Vh_*)$ which are sent to the subgroup  $F^p\Ah^n(X;\Vh_*)$ under $d \colon \Ah^{n-1}(X;\Vh_*) \to \Ah^{n}(X;\Vh_*)$. 


We will now introduce the cobordism relation.

\begin{defn}\label{def:GeometricBordismDatum} 
The group of geometric bordism data over $X$ is the subgroup of $\geocycles^n(\R\times X),$ with underlying maps $b=(c_b,f)\colon W\to \R\times X$ such that $0$ and $1$ are regular values for $c$.
Then $W_t=c_b^{-1}(t)$ is a closed manifold for $t=0,1$, and 
$f_t=f|_{W_t}$ is a geometric cycle. We define
\[
\partial\widetilde b := \geocycle_1-\geocycle_0 \in \geocycles^n(X)
\] 
and, setting $W_{[0,1]}=c_b^{-1}([0,1])$, we define
\begin{align*}
\psi^p(\widetilde b) &=
(-1)^{n}\left(f|_{W_{[0,1]}}\right)_* \left(K^p(\nabla_b)\right).
\end{align*}
\end{defn}

\begin{remark}
We will often write $\psi$ instead of $\psi^p$ to simplify the notation.
\end{remark}

\begin{prop} \label{geobordismequation}
For $\geocob$ a geometric bordism datum over $X$, we have
$$
\phi^p(\partial \geocob) - d\psi^p(\geocob) =0. 
$$
\end{prop}

\begin{proof}
Let $\sigma\in \Ah_c^*(X;\Vh_*)$. We use for $t=0,1$ the notation of the following diagram, where the square is cartesian, and where we view $j_t$ with the pullback orientation:
$$
\xymatrix{
W_{t}\ar[r]^{j_t} \ar[d]_{f_t} & \ar@/^2pc/[dd]^{f} W\ar[d]_-{(c,f)} \\
\{t\}\times X \ar[r]_{i_t}\ar[dr]_{\id} & \mathbb R\times X\ar[d]\\
& X
}
$$
Let $W_{[0,1]}=c^{-1}([0,1]).$ Since $K^p(\nabla_b)$ is closed and of even degree, we have $d(K^p(\nabla_b)\wedge f^*\sigma)=K(\nabla_b)\wedge df^*\sigma,$
and so by Stokes theorem, 
\begin{align*}
\left(f|_{W_{[0,1]}}\right)_*{K^p(\nabla_b)}(d\sigma)&= \int_{W_{[0,1]}} K(\nabla_b)\wedge df^*\sigma \\
&=\int_{\partial W_{[0,1]}} K(\nabla_b)\wedge f^*\sigma \\
&=\left(f|_{\partial W_{[0,1]}}\right)_*\left(\delta_{\partial W_{[0,1]}}\wedge K^p(\nabla_b)\right)(\sigma),
\end{align*} 
where $\delta_{\partial W_{[0,1]}}$ is the integration current of $\partial W_{[0,1]}$ with the boundary orientation.
We observe that $\delta_{\partial W_{[0,1]}} = (j_1)_*1-(j_0)_*1$, since the pullback orientation coincides with the boundary orientation at $1$, but not at $0$. 
We have
\begin{align*}
        (f|_{W_{[0,1]}})_*(\delta_{\partial W_{[0,1]}}\wedge K^p(\nabla_b))&= (f|_{W_{[0,1]}})_*((j_1)_*j_1^*K(\nabla_b)-(j_0)_*j_0^*K^p(\nabla_b))\\
        &= (f_1)_*K^p(\nabla_{f_1})-(f_0)_*K^p(\nabla_{f_0})\\
        &=\phi^p(\geocycle_1)-\phi^p(\geocycle_0).
\end{align*}
Since $T(d\sigma)=(-1)^{\deg T+1}dT(\sigma)$, for homogeneous currents $T$, by definition of the exterior derivative on currents on $X$, this finishes the proof.
\end{proof}

In light of Proposition \ref{geobordismequation}, we consider $(\partial \geocob,\psi^p(\geocob))$ as a Hodge filtered cycle of degree $(\codim b,p)$. We call such cycles \emph{nullbordant} and let $BMU_{\geo}^n(p)(X)\subset ZMU^n(p)(X)$ denote the subgroup generated by the nullbordant cycles. 

We define
\begin{align*}
\widetilde F^p\Ah^{n-1}(X;\Vh_*):= F^p\Ah^{n-1}(X;\Vh_*) + d\Ah^{n-2}(X;\Vh_*).    
\end{align*}

Then we define the group of Hodge filtered cobordism relations by
\begin{align}\label{eq:def_of_geometric_bordism_data}
BMU^n(p)(X)=BMU^n_{\geo}(p)(X) + a\left(\widetilde F^p\Ah^{n-1}(X;\Vh_*)\right)
\end{align}
where $a$ is the map defined in \eqref{cyclstructmaps} above.

Now we are ready to define geometric Hodge filtered cobordism: 

\begin{defn}\label{def:geometric_HFC_bordism_MUn(p)(X)}
Let $(X,F^*)\in\ManF$ and let $n$ and $p$ be integers. 
The geometric Hodge filtered cobordism group of $X$ of degree $(n,p)$ is defined as the quotient 
\[
MU^{n}(p)(X):=\frac{ZMU^{n}(p)(X)}{BMU^n(p)(X)}.
\]
We denote the Hodge filtered cobordism class of $\gamma=(\geocycle,h)\in ZMU^n(p)(X)$
by $[\gamma]=[\geocycle,h]$.
\end{defn}


\subsection{The long exact sequence}

The maps defined on the level of cycles in \eqref{cyclstructmaps} induce maps on the level of cohomology groups: 

\begin{alignat}{2}\label{structuremapsGeoHFCBordism}
\nonumber	& R \colon MU^n(p)(X)\to H^n(X;F^p\Ah^*(\mathcal V_*)),   & & R[\geocycle, \omega, h]=[\omega]\\
 & a \colon  H^{n-1}\left(X;\frac{\Ah^*}{F^p}(\Vh_*)\right) \rightarrow MU^n(p)(X),\quad \quad  &&a(h) = [0,dh,h]\\
\nonumber & I \colon ZMU^n(p)(X)\rightarrow MU^n(X),    & & I[\geocycle ,h,\omega]= [f].
\end{alignat} 

\begin{prop}\label{prop:RaI_are_well_defined}
The maps $R$, $a$ and $I$ in  \eqref{structuremapsGeoHFCBordism} are well-defined.
\end{prop}
\begin{proof}
We first show that $I$ and $R$ vanish on $BMU^n(p)(X)$, as defined in \eqref{eq:def_of_geometric_bordism_data}. 
For $\hfcycle=(\partial\geocob, \psi(\geocob))\in BMU^n_{\geo}(X),$ we have
$$
I(\hfcycle)=\partial b \in BMU^n(X),\quad \text{ and } \quad R(\hfcycle)=0,
$$
where the second equality is Lemma \ref{geobordismequation}, and $b$ is the bordism datum underlying $\geocob$. We have $I\circ a=0$, so in particular 
$$
I(a(h))=0,\quad h\in \widetilde F^p\Ah^{n-1}(X;\Vh_*),
$$
which finishes the proof that $I$ is well-defined. We have
$$
R\circ a\left(\widetilde F^p\Ah^{n-1}(X;\Vh_*)\right)=d\left(F^p\Ah^{n-1}(X;\Vh_*)\right)
$$ 
which is the group of relations for 
$$
H^n(X;F^p\Ah^*(\mathcal V_*))\simeq \frac{F^p\Ah^n(X;\Vh_*)_{\cl}}{dF^p\Ah^{n-1}(X;\Vh_*)}
$$ 
so $R$ is well-defined too. 
That $a$ is well-defined follows from the isomorphism  
\begin{align*}
H^n\left(X;\frac{\Ah^*}{F^p}(\Vh_*)\right) \cong \frac{d^{-1}(F^p\mathcal A^{n+1}(X;\Vh_*))^{n}}{\widetilde F^p\mathcal A^{n}(X;\Vh_*)}
\end{align*} 
and the definition of $BMU^n(p)(X)$.
\end{proof}

\begin{remark}\label{rem:diagram_for_axioms}
It is clear that $R\circ a=d$, and by construction we have 
$$
[R(\hfcycle)] = \phi^p(I(\hfcycle)) \quad \text{in}\ H^n(X;\Vh_*).
$$
Hence the diagram
\[
\xymatrix{
MU^n(p)(X)\ar[r]^I\ar[d]_{R} & MU^n(X)\ar[d]^{\phi^p}\\
H^n(X;F^p\Ah^*(\Vh_*))\ar[r]^-{\mathrm{inc}_*} & H^n(X;\Vh_*).
}
\]
commutes, where $\mathrm{inc}_*$ is the map induced by the inclusion of complexes of sheaves $F^p\Ah^*(\Vh_*)\xto{\mathrm{inc}} \mathcal A^*(\Vh_*)$. 
\end{remark}


Let $\overline{\phi}$ denote the composition of $\phi$ with the reduction modulo $F^p$ map:
\[
\overline{\phi} =\left( MU^n(X)\xto{\phi} H^n(X;\Ah^*(\Vh_*)) \to H^{n}\left(X;\frac{\Ah^*}{F^p}(\Vh_*)\right)\right).
\]

\begin{theorem}
\label{longexactseq}
There is a long exact sequence:
\begin{align*}
    \xymatrix @R=0pc{
    \cdots\ar[r] & H^{n-1}\left(X;\frac{\Ah^*}{F^p}(\Vh_*)\right) \ar[r]^-a & MU^n(p)(X)\ar[r]^-I& \\
    MU^n(X)\ar[r]^-{\overline{\phi}}& H^{n}\left(X;\frac{\Ah^*}{F^p}(\Vh_*)\right) \ar[r]^-a & MU^{n+1}(p)(X)\ar[r]&\cdots
    }
\end{align*}
\end{theorem}
\begin{proof}
We start with exactness at $MU^n(p)(X)$. First we observe 
$$
I(a([h]))=I([0,dh,h])=0.
$$ 
The converse requires more work. 
We work at the cycle level, so let $\hfcycle=(\geocycle,\omega,h)\in ZMU^n(p)(X)$ and suppose $I(\gamma)=0$. 
That means $f=\partial b$ for some bordism datum $b$. 
We may extend the geometric structure of $\geocycle$ over $b$ and obtain a geometric bordism datum $\widetilde b$ such that $\partial \widetilde b=\geocycle$. We have 
$$
(\geocycle, \omega,h)-(\partial\widetilde b, 0,\psi(\widetilde b)) = (0, \omega,h')=a(h').
$$
The last equality follows from the observation that, since $(0,\omega,h')\in ZMU^n(p)(X)$ is a Hodge filtered cycle, we must have
$$
dh'=\omega\in F^p\Ah^n(X;\Vh_*).
$$
Hence we know $\gamma \in BMU^n(p)(X)$.
Next we show exactness at $MU^n(X)$. 
The vanishing $\overline{\phi}\circ I = 0$ follows from the following commutative diagram, where the bottom row is exact:
$$
\xymatrix{
 MU^n(p)(X)\ar[d]_R \ar[r]^-I & MU^n(X)\ar[d]_-\phi\ar[dr]^-{\overline{\phi}} & \\
H^{n}(X;F^p\Ah^*(\Vh_*))\ar[r]^-{\inc_*} & H^n(X;\Vh_*)\ar[r] & H^{n}\left(X;\frac{\Ah}{F^p}(\Vh_*)\right).
}
$$
Conversely, suppose $\overline\phi([f])=0$. Then we can find $\omega\in F^p\Ah^n(X;\Vh)$ such that 
$$
\phi([f]) = \inc_*([\omega]).
$$
Let $\nabla_f$ be a connection on $N_f$ so that we get a geometric cycle $\geocycle$ with $I(\geocycle)=f$. 
Then $\phi(\geocycle)$ is a current representing $\phi([f])$. 
Hence $\phi(\geocycle)$ and $\omega$ are cohomologous, i.e., there is a current $h\in \Ds^{n-1}(X;\Vh_*)$ such that $\phi(\geocycle)- dh=\omega$. 
Then $\hfcycle:=(\widetilde f, \omega,h)$ is a Hodge filtered cycle with $I(\hfcycle)=f$.

Now we show exactness at $H^{n}\left(X;\frac{\Ah}{F^p}(\Vh_*)\right)$. 
Let $f\colon Z\to X$ be a bordism cycle on $X$. We will show $a(\overline\phi([f])) = 0\in MU^{n+1}(p)(X)$.  
Lifting $f$ to a geometric cycle $\geocycle\in\geocycles^n(X)$ we can write
$$
a(\overline\phi([f])=[0,0,\phi(\geocycle)].
$$
We may build from $\geocycle$ a geometric bordism datum $\geocob$ with underlying map
$$
\xymatrix{Z\ar[r]^-{(\frac{1}{2}, f)}&\mathbb R\times X}
$$
where $\frac{1}{2}$ denotes the constant map with value $\frac{1}{2}$.  
Clearly $\partial\geocob =0$. More interesting is the observation that $\psi(\geocob)=(-1)^{n}\phi(\geocycle)$. Hence
$$
(\partial \geocob,0, \psi(\geocob)) = (0,0,(-1)^{n} \phi(\geocycle)) \in BMU^n_{\geo}(X)
$$
and we conclude that $a(\overline\phi([f]))=0$.

Conversely, suppose that $h\in (d^{-1}F^p\Ah^n(X;\Vh_*))^{n-1}$ is such that $a(h)=(0,dh,h)$ represents $0$ in $MU^n(p)(X)$. 
Then there must be a geometric bordism datum $\geocob$ with underlying map $(c_b,f_b)\colon W\to \mathbb R\times X$, and a form $h'\in \widetilde F^p\Ah^{n-1}(X;\Vh_*)$ such that 
$$
(0,dh,h)=(\partial\geocob,dh',\psi(\geocob)+h').
$$ 
Since $\partial\geocob=0$, we have that 
$$
f:=f_b|_{c_b^{-1}([0,1])}\in ZMU^n(X)
$$
is a bordism cycle. 
By definition of $\psi$, we have $\psi(\geocob)=(-1)^n \phi(\geocycle)$ where $\geocycle$ is the obvious geometric cycle over $f$. We now have the following computation in $\HXAFp$:
$$
[h]=[h-h']=[\psi(\geocob)]=(-1)^n \overline{\phi}([f])\in \Imm(\overline{\phi}).
$$
This finishes the proof.
\end{proof}


\subsection{Pullbacks}

We now establish the contravariant functoriality of $MU^n(p)(X)$, along the lines of \cite{Bunke2009}[sections 4.2.5--4.2.6].
We denote by $WF(u)$ the wave-front set of a current $u$ in the sense of \cite[8.1]{Hoer}. See in particular \cite[Def.\,8.1.2, and p.\,265]{Hoer} for the definition. 
For a smooth map $f\colon Z\to X$, we denote by $N(f)\subset T^*X$ the normal set of $f$: 
$$
N(f)=\{v\in T^*X : v \neq 0,\ f^*v=0\}.
$$
Let $g \colon Y \to X$ be a morphism in $\ManF$. We define $ZMU_g^n(p)(X)$ to be the subset of $ZMU^n(p)(X)$ consisting of those $\gamma=(\geocycle,h)$ satisfying
\begin{itemize}
    \item $WF(h)\cap N(g)=\emptyset$, and
    \item $g\pitchfork f$.
\end{itemize}
For $\gamma = (\geocycle,h) \in ZMU_g^n(p)(X)$, we define $g^*\gamma\in ZMU^n(p)(Y)$ by
\begin{equation}
\label{defpb}
g^*(\hfcycle)=(g^*\geocycle, g^*h).
\end{equation}
Here $g^*h$ is defined by \cite[Theorem 8.2.4]{Hoer}. 
To define $g^*\geocycle$ we use the transversality property and consider the following cartesian diagram of manifolds:
$$
\xymatrix{g^*Z\ar[r]^{G}\ar[d]_{g^*f} & Z\ar[d]^f\\ Y\ar[r]_g & X.}
$$
Then $g^*f$ is complex-oriented, with $N_{g^*f}=G^*N_f$. We define $g^*\geocycle$ by
$$
g^*\geocycle = (g^*f, G^*N_f, G^*\nabla_f).
$$

The aim of this section is to show the following theorem: 

\begin{theorem} \label{PullbackHFCBordism}
The above pullback construction induces a map 
$$
g^*\colon MU^n(p)(X)\to MU^n(p)(Y)
$$
for any morphism $g:Y\to X$, making $MU^n(p)$ a contravariant functor on $\ManF$.
\end{theorem}

The proof proceeds in three steps:
\begin{prop}
\label{pbwd}
Given a morphism $g \colon Y\to X$ in $\ManF$, and a Hodge filtered cycle $\gamma\in ZMU^n(p)(X),$ there exist $b\in BMU^n(p)(X)$ such that 
$$
\gamma+b\in ZMU_g^n(p)(X).
$$ 
\end{prop}

\begin{prop}
\label{invpb}
If $\gamma\in ZMU_g^n(p)(X)\cap BMU^n(p)(X)$, 
then 
\[
g^*\hfcycle \in BMU^n(p)(Y).
\]
\end{prop}

For morphisms $g_1\colon X_1\to X_2$ and $g_2\colon X_2\to X_3$ in $\ManF$, we define 
\begin{equation*}
ZMU^n_{g_1g_2}(p)(X_3) := ZMU^{n}_{g_2}(p)(X_3)\cap ZMU^n_{g_2\circ g_1}(p)(X_3).
\end{equation*}

\begin{prop}\label{NaturalityOfZMUnpPullback}
We have $g_2^*\left(ZMU^n_{g_1g_2}(p)(X_3)\right) \subset ZMU^n_{g_1}(p)(X_2)$, and
$$
g_1^*\circ g_2^* = (g_2\circ g_1)^*\colon ZMU^n_{g_1g_2}(p)(X_3) \to ZMU^n(p)(X_1).
$$
\end{prop}

Since transversality is a generic property for smooth maps, and homotopies can be viewed as cobordisms, these three propositions together prove Theorem \ref{PullbackHFCBordism}.
\begin{proof}[Proof of Proposition \ref{NaturalityOfZMUnpPullback}]
Let $\hfcycle=(\geocycle, h)\in ZMU^n_{g_1g_2}(p)(X)$. The equality of currents $g_1^*g_2^*h=(g_1\circ g_2)^*h$ follows from \cite[Theorem 8.2.4]{Hoer}. 
We only need to add that connections pull back in a natural way as well.
\end{proof}

\begin{proof}[Proof of Proposition \ref{pbwd}]
By Thom's transversality theorem  
we may choose $f_1$ homotopic to $f$ so that $f_1\pitchfork g$. 
Then there is a complex orientation of $f_1$ and a cobordism between $f_1$ and $f$ of the form $b=\mathbb R\times Z\to \mathbb R\times X$ with $b(t,z)=(t,H(t,z))$ where $H$ is a homotopy between $f_1$ and $f$.
We may extend the geometric structure of $\widetilde f$ over $b$ and obtain a geometric bordism datum $\widetilde b$.
We give $f_1$ the geometric structure $i_1^*(\geocob)$, for $i_1\colon X\to \mathbb R\times X$ the inclusion $i_1(x)=(1,x)$. By design we have $\partial \geocob=\widetilde f_1-\widetilde f$. Hence 
$$
(\widetilde f_1, \omega, h+ \psi(\geocob)) - (\widetilde f, \omega, h)\in BMU^n_{\geo}(X),
$$
and $g^*\widetilde f_1$ is well-defined. 
Since $\phi(\widetilde f)- dh$ is smooth, we must have $WF(\phi(\widetilde f))=WF(dh)$. Since $d$ is a differential operator we have $WF(dh)\subset WF(h)$ by \cite[8.1.11]{Hoer}, but equality need not hold. 
To complete the proof, we must show that upon replacing $h$ with a cohomologous current if necessary, we can assume $WF(h) = WF(dh)$, since $a(d\beta) \in BMU^n(p)(X)$ by definition. 
Thus we have reduced the proof of the assertion to proving the following lemma.
\end{proof}

\begin{lemma}
\label{4.11}
Let $\alpha\in \Ds^{n}(X;\mathcal V_*)$. 
Then there exists a current $\beta\in \Ds^{n-1}(X;\mathcal V_*)$ so that $WF(d\alpha)=WF(\alpha+d\beta)$.
\end{lemma}

\begin{proof}
This is \cite[Lemma 4.11]{Bunke2009}. For convenience of the reader, we recount their proof. Choose a Riemannian metric on $X$. Let $d^*$ be formally adjoint to $d$. Then we consider the Laplacian $\Delta=d^*d+dd^*$, which is an elliptic differential operator
$$
\Delta\colon \Ds^*(X)\to \Ds^*(X).
$$
Using \cite[Theorem 18.1.24]{HoerIII} we can find a parametrix $\mathscr P\colon \Ds^*(X)\to \Ds^*(X)$, properly supported in the sense that both projections from the support of the Schwartz kernel of $\mathscr P$, which we denote $P$, $\Ds^*(X\times X)\supset \supp P \to X $ are proper maps, such that both $\Delta\circ \mathscr P-id$ and $\mathscr P\circ \Delta - id$ are smoothing operators. We put $G=d^*\mathscr P$. This pseudo-differential operator satisfies 
$$
dG+Gd=1+S
$$ 
for a smoothing operator $S$. Let $\beta= G\alpha$. 
Then we get 
\begin{align*}
    \alpha- d\beta &= \alpha - dG\alpha 
    = \alpha - (1-G d+ S)\alpha 
    =-G d\alpha  +S\alpha.
\end{align*}
Since $G$ is a pseudo-differential operator, we have $WF(G d\alpha)\subset WF(d\alpha)$. 
This can, for example, be seen by taking $\Gamma=T^*X\backslash 0$ in \cite[Proposition 18.1.26.]{HoerIII}. This finishes the proof.
\end{proof}

\begin{proof}[Proof of Proposition \ref{invpb}]
Let $\hfcycle=(\geocycle, h)\in ZMU^n(p)(X)_g\cap BMU^n(p)(X)$. We can then write 
$$
\hfcycle  = (\partial \widetilde b, \phi(\widetilde b))+a(h)
$$ where $\geocob$ is a geometric cobordism with underlying map $ b=(a_b,f_b):W\to \mathbb R\times X$ and 
$h\in \widetilde F^p\Ah^{n-1}(X;\Vh_*)$. 
We must show that $g^*\hfcycle \in BMU^n(p)(Y)$. 
We start by noting that, since $g$ is a morphism in $\ManF$, we have
$$
g^*a(h)=a(g^*h)\in BMU^n(p)(Y).
$$ 
We can perturb $f_b$ slightly, without $\partial\geocob$, so as to ensure $f_b\pitchfork g$. 
We extend the geometric structure of $\geocob$ over the perturbing homotopy 
and obtain a geometric bordism datum  $\geocob$ with $\partial \geocob=\geocycle$.
Now we consider the pullback geometric bordism datum $g^*\geocob$, which is formally defined as the geometric cycle $(id_{\mathbb R} \times g)^*\widetilde b$. It has underlying map $(a',f')\colon W'\to Y$, fitting into the following diagram where all squares are pullback squares of manifolds for $t\in \{0,1\}$.
$$
\xymatrix{
  & W_t\ar[d]\ar[dr] & \\
W'_t\ar[d]\ar[dr]\ar[ur] & t\times X\ar[dr] & W\ar[d]^-{(a_b,f_b)} \\
t\times Y\ar[dr]\ar[ur] & W'\ar[d]^(.3){(a'_b,f'_b)}\ar[ur]_(.3){G} & \mathbb R\times X \\
 & \mathbb R\times Y \ar[ur]_{id_\mathbb R\times g} & 
}
$$
It is then clear that $\partial g^* \widetilde b = g^*\partial\widetilde b$. We get from Theorem \ref{currentialPushPull} below that 
\begin{align*}
g^*\psi(\widetilde b) &= (-1)^ng^*(f_b|_{a_{b}^{-1}([0,1])})_*K(\nabla_{ b}) \\
    &= (-1)^n\left(f'_b|_{a_b^{\prime -1}([0,1])}\right)_*G^*K(\nabla_b) \\
        &= \psi(g^*\geocob). 
\end{align*}
Hence we have 
\[
g^*(\partial \geocob, \psi(\geocob))=(\partial g^*\geocob, \psi(g^*\geocob))\in BMU^n(p)(Y). \qedhere
\]
\end{proof}

It remains to show the following result which is certainly well-known. 
However, we were unable to find a reference and therefore provide a proof.

\begin{theorem}\label{currentialPushPull}
Let $f$ be a proper, oriented smooth map of codimension $d$, and let $g$ be smooth with $f\pitchfork g$. 
We consider the cartesian diagram
$$
\xymatrix{
Z':= Z\times_X Y\ar[r]^-{g'}\ar[d]_-{f'} & Z\ar[d]^-f \\
Y\ar[r]_-g & X.
}
$$
Then the following diagram exists and commutes:
$$
\xymatrix{
\Ah^*(Z')\ar[d]_-{f'_*} & \ar[l]_-{g^{\prime *}} \Ah^*(Z)\ar[d]^-{f_*} \\
\Ds^{*+d}(Y) & \ar[l]^-{g^*}  \Ds_g^{*+d}(X).
}
$$
Here we mean by $\Ds^*_g(X)$ the currents $T$ on $X$ which may be pulled back along $g$, in the sense that there is a unique current, which we label $g^*T$, in the set 
$$
\left \{\lim_{t\to 0} g^*(\omega_t) \ :\ \lim \omega_{t\to 0} = T \right \}.
$$
In particular, we have an equality of continuous $\C$-linear maps
$$
g^*\circ f_* = f'_*\circ g'^*\colon \Ah^*(Z)\to \Ds 
^{*+d}(Y).
$$
\end{theorem}
\begin{proof}
We will apply micro-local analysis and the Schwartz kernel theorem \cite[p.\,93--94]{HoerIII}.  
The Schwartz kernel theorem states that there is a bijection between $\Ds^*(Z\times X)$ and the space of maps $\Ah^*_c(Z)\to \Ds^*(X)$ given by 
$$
\Ds^*(Z\times X)\ni T\mapsto \left(\sigma_Z\mapsto \left(\sigma_X\mapsto T(\sigma_Z\otimes \sigma_X)\right)\right).
$$
Then $T$ is called the Schwartz kernel of the mapping it corresponds to.
We need the following three facts:
\begin{enumerate}
\item The Schwartz kernel of $f_*$ is the integration current $\delta_{\Graph(f)}$, where 
$$
\Graph(f)=\{(z,f(z))\}\subset Z\times X.
$$
\item The Schwartz kernel of $g^*$ is the integration current $\delta_{\Graph'(g)}$ of the transposed graph
$$
{\Graph}^{\prime}(g)=\{(g(y),y)\}\subset X\times Y.
$$
\item The kernel of $f^{\prime}_*\circ g^{\prime*}$ is the integration current of $(g',f')(Z')$.
\end{enumerate}
The last statement follows from
\begin{align*}
\delta_{(g',f')(Z')}(\sigma_Z\otimes \sigma_Y) &= \int_{Z'} g'^*\sigma_Z\wedge f'^*\sigma_X.
\end{align*}
The first two statements are special cases with either $f'$ or $g'$ as the identity. 

We now apply Hörmander's criterion, see \cite[8.2.14]{Hoer} and the preceding discussion. 
We get that when $f\pitchfork g$, then $g^*\circ f_*$ is continuous with kernel
\begin{align*}
K:=\left(\pi_{Z,Y}\right)_*\left(\id_Z\times \Delta \times \id_Y\right)^*\delta_{\Graph(f)}\otimes \delta_{{\Graph}'(g)},
\end{align*}
where $\pi_{Z,Y}\colon Z\times X\times Y\to Z\times Y$ is the projection.
Using \cite[Example 8.2.8]{Hoer} we can now deduce that $K$ is the current $(g',f')_*1_{Z'}$. Since we saw already that this is the kernel of $f'_*\circ g'^*$, we get from the Schwartz kernel theorem that 
\[
f'_*\circ g'^* = g^*\circ f_*. \qedhere
\]
\end{proof}

    
\subsection{Products}
\label{subsec:products}

We first define exterior products 
\begin{align}
    \label{extprod}
    \times\colon MU^{n_1}(p_1)(X_1)\times MU^{n_2}(p_2)(X_2)& \to MU^{n_1+n_2}(p_1+p_2)(X_1\times X_2)
    \end{align}
on where $X_1\times X_2$ carries the filtration of Definition \ref{def:Hodge_filtration_on_product}. When $X$ satisfies condition \eqref{productCondition}, i.e., when the diagonal $\Delta\colon X\to X\times X$ is a morphism in $\ManF$, we get the structure of a bigraded ring on 
$$
MU^*(*)(X)=\bigoplus_{n,p}MU^n(p)(X). 
$$

Let for $i=1,2$
$$
\hfcycle_i = (\geocycle_i,h_i)\in ZMU^{n_i}(p_i)(X_i)
$$ 
be Hodge filtered cycles with underlying maps $f_i\colon Z_i\to X_i$, and let $\pi_i$ be the projection $Z_1\times Z_2\to Z_i$. We define the exterior product of geometric cycles by 
$$
\widetilde f_1\times \widetilde f_2 = (f_1\times f_2,\ N_1\times N_2,\ \nabla_1\times \nabla_2),
$$ 
where we abbreviate $N_{f_1}$ by $N_1$ and so on, $N_1\times N_2 = \pi_1^*N_1 \oplus \pi_2^* N_2$, and $\nabla_1\times\nabla_2 = \pi_1^* \nabla_1 \oplus \pi_2^*\nabla_2$. 
We have the product 
$$\otimes \colon \Ds^{n_1}(X_1;\Vh_*)\times \Ds^{n_2}(X_2;\Vh_*)\to \Ds^{n_1+n_2}(X_1\times X_2;\Vh_*),
$$
satisfying $T_1\otimes T_2=\pi_1^*T_1\wedge \pi_2^*T_2$. Since $K$ is multiplicative, we have
$$
K^{p_1+p_2}(\nabla_1\times \nabla_2)=K^{p_1}(\nabla_1)\otimes K^{p_2} (\nabla_2).
$$
We now return to suppressing the $p$ in $K^p=(2\pi i)^p\cdot K$ and $\phi^p$ from the notation. 
Since $(f_1\times f_2)_*(T_1\otimes T_2)=(f_1)_*T_1\otimes (f_2)_*T_2$, we get
\begin{equation}
\phi(\geocycle_1\times\geocycle_2)=\phi(\geocycle_1)\otimes\phi(\geocycle_2).
\end{equation}
We want $R(\hfcycle_1\times\hfcycle_2) = R(\hfcycle_1)\otimes R(\hfcycle_2)$. We compute:
\begin{align*}
    \phi(\geocycle_1)\otimes \phi(\geocycle_2) &-R(\hfcycle_1)\otimes R(\hfcycle_2)  \\
    & =\phi(\widetilde f_1)\otimes \phi(\widetilde f_2)- (\phi(\geocycle_1)- dh_1)\otimes (\phi(\geocycle_2)- dh_2) \\
    & = dh_1\otimes \phi(\widetilde f_2) + \phi(\widetilde f_1)\otimes dh_2 -dh_1\otimes dh_2 \\
    &=d\left(h_1\otimes \phi(\widetilde f_2) +(-1)^{n_1}\phi(\widetilde f_1)\otimes h_2 - h_1\otimes dh_2\right) \\
    & = d\left(h_1\otimes R(\hfcycle_2) + (-1)^{n_1} \phi(\geocycle_1)\otimes h_2\right)
\end{align*}
Therefore we define the exterior product of Hodge filtered cycles by
\begin{equation}\label{exteriorProductHFcycleDef}
    \hfcycle_1\times \hfcycle_2 := \left( \widetilde f_1\times\widetilde f_2,\quad h_1\otimes R(\hfcycle_2) + (-1)^{n_1} \phi(\geocycle_1)\otimes h_2\right).
\end{equation}

\begin{remark} \label{remarkSymmetryExteriorProducthfcycles}
In the above computation, we chose $h_1\otimes dh_2$ as a current with exterior derivative $dh_1\otimes dh_2$. 
If we had chosen instead $(-1)^{n_1}dh_1\otimes h_2$, we would have been led to define 
\begin{equation*}
    \hfcycle_1\times \hfcycle_2 = \left( \widetilde f_1\times\widetilde f_2,\quad h_1\otimes \phi(\geocycle_2) + (-1)^{n_1} R(\hfcycle_1)\otimes h_2\right).
\end{equation*}
Since we work modulo $\mathrm{Im}(d)$, this choice is immaterial as we have
$$
h_1\otimes dh_2- (-1)^{n_1}dh_1\otimes h_2 = d\left(h_1\otimes h_2\right).
$$
\end{remark}

\begin{prop}
Let $s\colon X_1\times X_2\to X_2\times X_1$ be the swap map $s(x_1,x_2)=(x_2,x_1)$. The exterior product \eqref{exteriorProductHFcycleDef} satisfies
\begin{align*}
(\hfcycle_1+\hfcycle'_1)\times \hfcycle_2 &= \hfcycle_1\times\hfcycle_2+\hfcycle'_1\times \hfcycle_2,\\
\hfcycle_1\times \hfcycle_2 &= (-1)^{n_1}s^*\left(\hfcycle_2\times \hfcycle_1 
\right). 
\end{align*}
\end{prop}

\begin{proof}
The isomorphism of geometric cycles underlying the first equation is obvious. Then the first equation follows since the expression $h_1\otimes \phi(\geocycle_2) + (-1)^{n_1} R(\hfcycle_1)\otimes h_2$ is linear in $h_1$. The second equality follows from Remark \ref{remarkSymmetryExteriorProducthfcycles}. 
\end{proof}

We now show that the cobordism class of $\hfcycle_1\times \hfcycle_2$ depends only on the cobordism class of $\hfcycle_i$. Because of the symmetry, it suffices to show that if $\hfcycle_2$ represent $0$, i.e., $\hfcycle_2\in BMU^{n_2}(p_2)(X_2)$, then 
$$
\hfcycle_1\times\hfcycle_2\in BMU^{n_1+n_2}(p_1+p_2)(X_1\times X_2).
$$
We can write 
$$
\hfcycle_2=(\partial \geocob, \psi(\geocob)+h)=(\partial \geocob, \psi(\geocob))+a(h)
$$ 
for $\widetilde b$ a geometric bordism datum over $X_2$, with underlying map $(a_b,f_b)$, and $h\in \widetilde F^{p_2}\Ah^{n_2}(X_2;\Vh_*)$. 
We note first that since $\geocycle_1\times 0 = 0$ we get
\begin{align} \label{formulaForMultiplicativeIdentityMUnp}
\hfcycle_1\times a(h) &= (0,\quad h_1\otimes dh + \phi(\geocycle_1)\otimes h)\\
    \nonumber &= a\left(R(\hfcycle_1)\otimes h + d\left((-1)^{n_1+1}h_1\otimes h\right)\right).
\end{align}
Since
$$
R(\hfcycle_1)\otimes h - d\left((-1)^{n_1+1}h_1\otimes h\right)\in \widetilde F^{p_1+p_2}\Ah^{n_1+n_2}(X_1\times X_2;\Vh_*)
$$ 
we conclude that $\hfcycle_1\times a(h)$ represent $0$. We now suppose $\hfcycle_2=(\partial \geocob, \psi(\geocob))$. Then $R(\gamma_2)=0$, and we get
\begin{align*}
\hfcycle_1\times \hfcycle_2&=\gamma_1\times (\partial \geocob,\ \psi(\geocob)) \\
 &= (\geocycle_1\times \partial \geocob,(-1)^{n_1}\phi(\geocycle_1)\otimes \psi(\geocob)) \\
 &= (\partial(\geocycle_1\times\geocob), \psi(\geocycle_1\times\geocob))
\end{align*}
where we interpret $\geocycle_1\times\geocob$ as a geometric bordism datum on $X_1\times X_2$, and the sign is absorbed by the sign in the definition of $\psi$. We have now established the exterior product \eqref{extprod}.

\begin{defn}\label{def:GHFC_ring_structure}
We assume that $(X,F^*)$ satisfies condition \eqref{productCondition}, i.e., that $\Delta\colon X\to X\times X$ is a morphism in $\ManF$. 
Using \eqref{extprod} we turn $MU^*(*)(X)$
into a ring with the product
\begin{align}
    \label{intprod}
    MU^n(p)(X)\times MU^m(q)(X)& \to MU^{n+m}(p+q)(X) 
\end{align} 
defined by
$$
[\hfcycle_1]\cdot [\hfcycle_2] = \Delta^*([  \hfcycle_1\times \hfcycle_2]).
$$
\end{defn}


\subsection{Proposed axioms for Hodge filtered cohomology theories}\label{sec:axioms}

We end this section with a brief discussion of a general framework for Hodge filtered extensions of cohomology theories. 
It is inspired by the axiomatic approach of \cite{Bunke2010} to differential cohomology.

\begin{defn}\label{def:axioms}
Let $h^*$ be a cohomology theory for topological spaces and assume that $h^*$ is rationally even in the sense that $h^*(\pt)\otimes \mathbb Q$ is an evenly graded vector space. 
For $p\in\mathbb Z$ and an evenly graded $\C$-vector space $\Vh_*$, let 
\[
\phi^p \colon h^*\to H^*(-;\Vh_*)
\]
be a morphism of cohomology theories. 
Following \cite{Bunke2010}, we define a \emph{Hodge filtered extension} of $(h^*,\phi^p)$ over $\ManF$ to be a contravariant functor 
\[
h^*_\Dh(p) \colon \ManF \to \Ab^{\mathbb Z}
\]
together with natural transformations 
\begin{itemize}
    \item $a\colon H^{n-1}\left(X;\frac{\Ah^*}{F^p}(\Vh_*)\right)\to h_\Dh^n(p)(X)$
    \item $I\colon h_\Dh^n(p)(X)\to h^n(X)$
    \item $R\colon h_\Dh^n(p)(X)\to H^n\left(X;F^p\Ah^*(\Vh_*)\right)$
\end{itemize}
where we write $\frac{\Ah^*}{F^p}(\Vh_*)$ for the quotient of complexes of sheaves  $\frac{\Ah^*(\Vh_*)}{F^p\Ah^*(\Vh_*)}$. 
These data are required to satisfy the following conditions: 
\begin{itemize}
    \item The diagram 
        \begin{align*}
        \xymatrix{
        h_\Dh^n(p)(X)\ar[r]^-I\ar[d]_-R & h^n(X)\ar[d]^-{\phi^p}\\
        H^n(X;F^p\Ah^*(\Vh_*)) \ar[r]^-{\mathrm{inc}_*} & H^{n}(X;\Ah^*(\Vh_*)) 
        }
        \end{align*}
        commutes, where $\mathrm{inc}_*$ is the map induced by the inclusion of complexes of sheaves $F^p\Ah^*(\Vh_*)\xto{\mathrm{inc}} \mathcal A^*(\Vh_*)$. 

    \item $R\circ a= d$, where $d$  denotes the connecting homomorphism in cohomology which is induced by the differential. 
    \item \label{hfcexseqh} 
    The sequence
    \begin{align}
    \label{eq:long_es_intro}
    \xymatrix@R=0.5pc{
    \cdots\ar[r]&h^{n-1}(X)\ar[r]^-{\overline{\phi^p}}
    & H^{n-1}\left(X;\frac{\Ah^*}{F^p}(\Vh_*)\right) \ar[r]^-a& h^n_\Dh(p)(X)\ar[r]^-I& \\ 
    & h^n(X)\ar[r] & H^{n}\left(X;\frac{\Ah^*}{F^p}(\Vh_*)\right)\ar[r]&\cdots 
    }
    \end{align}
    is exact, where $\overline{\phi^p}$ is the composition
    \begin{align*}
    h^n(X) \xto{\phi^p}  H^n(X;\Ah^*(\Vh_*)) \longrightarrow   H^n\left(X;\frac{\Ah^*}{F^p}(\Vh_*)\right).
    \end{align*}
    \end{itemize}

A morphism of Hodge filtered extensions over $(h^*,\phi^p)$ is a natural transformation $\kappa$ which commutes with the respective structure maps. 
\end{defn}

\begin{example}\label{HFC_is_a_HF_extension}
It follows from Proposition \ref{prop:RaI_are_well_defined}, Remark \ref{rem:diagram_for_axioms}, Theorem \ref{longexactseq}, and Theorem \ref{PullbackHFCBordism} that the functor $X \mapsto MU^*(p)(X)$ is a Hodge filtered extension of $(MU^*,\phi^p)$ over $\ManF$.  
In fact, together with the product structure of Definition \ref{def:GHFC_ring_structure} geometric Hodge filtered cobordism is a multiplicative Hodge filtered extension in the following sense: 
It is clear that $I$ is multiplicative  
and it follows from the computations prior to \eqref{exteriorProductHFcycleDef} that $R$ is multiplicative as well. 
Moreover,  we have 
$$
a([h])\cdot \hfcycle = a([h\wedge R(\hfcycle)]).
$$
\end{example}

\begin{example}\label{ex:ED_is_a_HF_extension}
For $h^*=E^*$, represented by a rationally even spectrum $E$ and $\Vh_* = E_*\otimes_{\Z} \C$, the functor $X \mapsto \ED^*(p)(X)$ is a Hodge filtered extension of $(E^*,\phi^p)$ over the subcategory $\Manc$. 
The maps $I$ and $R$ are induced by canonical maps $\ED(p) \to \sing(E)$ and $\ED(p) \to H(F^p\Ah^*(\Vh_*))$. 
In order to obtain $a$, we remark that $\ED(p)$ is equivalent to the homotopy fiber of the induced map $\sing(E) \to H\left(\frac{\Ah^*}{F^p}(\Vh_*)\right)$ where we use that the Eilenberg--MacLane functor $H$ 
is a Quillen  equivalence between stable model categories by \cite{shipley} and preserves homotopy pushouts. 
We refer to \cite[Chapter 3]{haus} for further details. 
In particular, Deligne cohomology is a Hodge filtered extension of singular cohomology. 
\end{example}

\begin{remark}
In \cite{Bunke2010} it is shown that under mild assumptions axioms analogous to the ones in Definition \ref{def:axioms} suffice to characterize differential extensions of cohomology theories up to isomorphism. 
One may therefore wonder whether the axioms of Definition \ref{def:axioms} also suffice to characterize Hodge filtered cohomology up to isomorphism. 
The obstruction to a translation of the proof from \cite{Bunke2010} to Hodge filtered extensions is essentially the fact that morphisms in $\ManF$ may be more sparse  than arbitrary smooth maps. 
For example, for complex manifolds it is  well-known that holomorphic maps are much more rigid than smooth maps. 
This has the following consequence. 
While the set of maps to a space which represents a cohomology theory can be approximated by smooth maps to suitable smooth manifolds, it cannot, in general, be  approximated by holomorphic maps to complex manifolds. 
Using Oka theory, see for example \cite{Lar1} or \cite{Forstneric}, we can, however, obtain partial results for the subcategory of Stein manifolds in $\Manc$. 
The underlying idea goes back to Gromov \cite[0.7.B.]{Gromov} who suggests to use Oka manifolds to encode topological information in holomorphic terms. 
The assumption we have to make is that the underlying cohomology theory $h$ can be represented by spaces which can be approximated by complex manifolds which are both Oka and Stein. 
Our only example of such a theory, however, is complex $K$-theory and we will therefore not include the argument in the present paper. 
We conclude this section with the remark that it is an open problem in complex analysis whether the homotopy types representable by Oka--Stein manifolds are the same as those representable by smooth manifolds. 
\end{remark}


\section{Homotopical model via Mathai--Quillen Thom forms}
\label{sec:homotopical_model}

In this section we prepare to apply the Pontryagin--Thom construction by giving a new homotopy-theoretical description of Hodge filtered complex cobordism for every $(X,F^*)\in \ManF$. 

\subsection{Geometry of the tautological bundles and compatibility}
We denote by $\trivbc_X$ and $\trivbr_X$ the trivial complex and real line bundles over $X$, or simply by $\trivbc$ and $\trivbr$ when the base space is evident. We consider these with their standard metrics and the connection $d$, which in each case is compatible with the metric. Let $\Gr_m(\C^{m+l})$ denote the Grassmannian of $m$-planes in $\C^{m+l}$. The tautological bundle $\gamma_{m,l}$ is defined by
$$
\gamma_{m,l}=\{(v,V)\in \C^{m+l}\times\Grml\ | \ v\in V\}\subset \trivbc^{m+l},
$$
Then $\gamma_{m,l}$ inherits a Hermitian metric $h_{m,l}$, and compatible connection $\nabla_{m,l}$. 

\begin{remark}
These connections are the same as those considered by Narashimhan--Ramanan in \cite{NR1}. They are there shown to be universal among unitary connections. 
\end{remark}

The various tautological bundles are connected by a system of maps induced by the inclusion $\C^{m+l}\into\C^{m+l+1}$, $(z_1,\cdots, z_{m+l})\mapsto (z_1,\cdots, z_{m+l},0)$ and the bijection $\C^{m+l}\times \C\to \C^{m+l+1}$, $((z_1,\dots, z_{m+l}),t)\mapsto (z_1,\dots, z_{m+l},t)$. The inclusion gives the right hand one of the following commutative diagrams, and the bijection gives the left hand one:
\begin{equation}\label{definkjnk}
\xymatrix@C=1.5em{
    \gamma_{m,l}\oplus \trivbc \ar[r]^-{\overline{j_{m,l}}} \ar[d] & \gamma_{m+1,l}\ar[d]    & \gamma_{m,l}\ar[r]^{\overline{i_{m,l}}}\ar[d] & \gamma_{m,l+1}\ar[d]\\
    \Gr_{m}(\C^{m+l})\ar[r]_-{j_{m,l}} & \Gr_{m+1}(\C^{m+1+l}) & \Gr_m(\C^{m+l})\ar[r]_{i_{m,l}} & \Gr_{m}(\C^{m+l+1}).
}
\end{equation}
Both diagrams \eqref{definkjnk} are cartesian, and $\overline{i_{m,l}}$ and $\overline{j_{m,l}}$ are bundle maps, i.e., continuous fiberwise linear isomorphisms. 

\begin{prop}
\label{canonicalConnectionsCompatible}
The connections $\nabla_{m,l}$ are compatible in the sense that
\[ \overline{i_{m,l}}^*\nabla_{m,l+1}=\nabla_{m,l},\quad \text{and}\quad \overline{j_{m,l}}^*\nabla_{m+1,l} = \nabla_{m,l}\oplus d.
\]
Here $d$ denotes the exterior derivative, thought of as a connection on the trivial bundle. 
\end{prop}

\begin{proof}
There is a map $\Gr_m(\C^{m+l})\to \Gr_{l}(\C^{m+l})$ given by $V\mapsto V^\perp$. We denote by $\perp$ the bundle map $\gamma_{m,l}^\perp\to\gamma_{l,m}$ given by $(v,V)\mapsto (v,V^\perp)$. This is a diffeomorphism.   
The bundle map
$$
\overline{i_{m,l}}\oplus (\perp^{-1}\circ \overline{j_{l,m}}\circ(\perp\oplus \id_{\trivbc}))\colon \gamma_{m,l}\oplus (\gamma_{m,l}^\perp\oplus \trivbc ) \to \gamma_{m,l+1}\oplus\gamma_{m,l+1}^\perp = \trivbc^{m+l+1}_{\Gr_{m}(\C^{m+l+1})}
$$
equals the map 
$$
\gamma_{m,l}\oplus\gamma_{m,l}^\perp\oplus\trivbc = \trivbc^{m+l+1}_{\Gr_{m}(\C^{m+l})}\to \trivbc^{m+l+1}_{\Gr_{m}(\C^{m+l+1})}
$$
given by $((v,V),(u,V),(V,t))\mapsto (V,v+u+t\cdot e_{m+l+1})$, where we use the inclusion $\C^{m+l}\to \C^{m+l+1}$ from above to view $v$ and $u$ as elements of $\C^{m+l+1}$, and $e_i$ is the $i$-th standard basis vector.  
This proves both claims: For the first claim, we  observe that the connection $\nabla_1$ induced on $E_1$ from a connection $\nabla$ on $E_1\oplus E_2$, is also induced on $E_1$ from $E_1\oplus E_2\oplus \trivbc$ with the connection $\nabla\oplus d$. 
For the second claim, we observe that $\nabla\oplus d$ induces on $E_1\oplus\trivbc$ the connection $\nabla_1\oplus d$.
\end{proof}

The above proof also shows that the metrics are compatible:

\begin{prop}\label{canonicalMetricsAreCompatible}
For varying $m$ and $l$, 
the Hermitian metrics $h_{m,l}$ on $\gamma_{m,l}$ are compatible in the sense that $\overline{j_{m,l}}$ and $\overline{i_{m,l}}$ are metric preserving bundle maps. 
\end{prop}

Hence $\gamma_{m,l}$ is canonically a Hermitian bundle with unitary connection, and these structures are compatible for various values of $m$ and $l$.


\subsection{Thom spaces, rapidly decreasing forms and fundamental forms}

For a vector bundle $E\to X$ over a compact base we consider the Thom space of $E$, denoted $\Thom(E)$, to be the one-point compactification of $E$. 
In general, $\Thom(E)$ is the colimit 
$$
\Thom(E):=\colim_{K\subset X}\Thom(E|_{K})
$$
over compacta $K\subset X$. As a set we have $\Thom(E)=E\sqcup \{\infty\}$ and the canonical inclusion $E\to \Thom(E)$. 
We view $\Thom(E)$ as a pointed space with $\infty$ as basepoint. 
We view suspensions as Thom spaces, $\Sigma X_+=\Thom(\trivbr_X)$. We consider $\Thom$ as a functor from the category of vector bundles and continuous fiberwise linear maps, to pointed topological spaces. 

Now let $X$ be a smooth manifold and $E\to X$ be a Euclidean vector bundle, let $D$ denote the open unit disc bundle of $E$, and let $\Phi\colon E\to D$ be the fiberwise diffeomorphism $\Phi(v)=v/\sqrt{1+|v|^2}$. 
Then we follow \cite[page 98]{mathai-quillen} and define the space of rapidly decreasing forms and currents by
\begin{align*}
\Ah_{rd}^n(E;\Vh_*)&:=\{\Phi^*\omega\ |\ \omega \in \Ah^n(E;\Vh_*),\ \text{and}\ \supp(\omega)\subset \overline{D}\} \\
\Ds_{rd}^n(E;\Vh_*)&:=\{\Phi^*\omega\ |\ \omega \in \Ds^n(E;\Vh_*),\ \text{and}\ \supp(\omega)\subset \overline{D}\}.
\end{align*}
As in \cite{mathai-quillen} an analysis of growth conditions is not necessary for our purposes. 
There are maps
\begin{align*}
\left(\pi^{\Thom(E)}_X\right)_* &\colon \Ds_{rd}^n(E;\Vh_*)\to \Ds^{n-\dim_\R E}(X;\Vh_*),\\ 
\int_{\Thom(E)/X} &\colon \Ah_{rd}^*(E;\Vh_*)\to \Ah^{*-\dim_\R E}(X;\Vh_*)
\end{align*}
for $\pi_X^E\colon E\to X$ the projection. 
Note that $\left(\pi^E_X\right)_*$ is defined since $\pi_X^E=\pi_X^E\circ \Phi$ and $\pi_X^E|_{\supp(\Phi_*\omega)}$ is proper for $T \in \Ds_{rd}^*(E;\Vh_*)$. 
The integration map $\int_{\Thom(E)/X}$ is the restriction and co-restriction of $\left(\pi_X^E\right)_*$, as in Remark \ref{IntegrationOverTheFiber}.

\begin{remark}
We note that, by the Thom isomorphism, the complex $\Ah^*_{rd}(E;\Vh_*)$ computes the reduced cohomology groups of $\Thom(E)$. 
We will use the notation 
\[
\Ah^*(\Thom(E);\Vh_*):=\Ah_{rd}^*(E;\Vh_*).
\]
\end{remark}

Mathai and Quillen constructed in \cite{mathai-quillen} rapidly decreasing \emph{Thom forms}, depending naturally on a Hermitian metric and unitary connection. See also \cite[\S 1.6]{BGV}. 
Let 
$$
MU(m,l):=\Thom(\gamma_{m,l}).
$$
We get Thom forms
$$
U_{m,l}\in \Ah^{2m}(MU(m,l);\Vh_*),
$$ 
which in light of Propositions \ref{canonicalConnectionsCompatible} and \ref{canonicalMetricsAreCompatible} are compatible in the sense of the ensuing proposition. Let $U_\C$ be the Mathai--Quillen Thom form of $\trivbc_{\pt}$. 

\begin{prop}\label{MathaiQuillenThomFormsCompatible}
We have $\overline{i_{m,l}}^*U_{m,l+1} = U_{m,l}$ and $\overline{j_{m,l}}^*U_{m+1,l} = U_{m,l}\otimes  U_\C$. \qed
\end{prop}

Now we are ready to define forms on $MU(m,l)$ which will induce a fundamental cocycle. We set
\begin{align}\label{definitionPhiml}
\phi_{m,l}= U_{m,l} \wedge \pi_{m,l}^*K(\nabla_{m,l})\in\Ah^{2m}(MU(m,l);\Vh_*)
\end{align}
where $\pi_{m,l}\colon \gamma_{m,l}\to \Gr_m(\C^{m+l})$ is the projection. Combining Propositions \ref{MathaiQuillenThomFormsCompatible} and \ref{canonicalConnectionsCompatible} we conclude:

\begin{prop}\label{phimlAreCompatible}
The forms $\phi_{m,l}$ satisfy the equalities  
\begin{align*}
\overline{i_{m,l}}^*\phi_{m,l+1} = \phi_{m,l} ~ \text{and} ~ \overline{j_{m,l}}^*\phi_{m+1,l} = \phi_{m,l}\otimes U_\C. \qed
\end{align*}
\end{prop}


\subsection{Thom spectra and the map \texorpdfstring{$A$}{A}}

We recall that $MU$ is the spectrum obtained from the spaces $MU(m,l)$ via the structure maps 
$$
s_{m,l}:= \Thom(\overline{i_{m,l}})\quad \text{and} \quad q_{m,l}:=\Thom(\overline{j_{m,l}})
$$
in the following way. First define the colimit $MU(m)=\colim_l MU(m,l)$ along the maps $q_{m,l}$. Then we get maps 
$$
s_m:=\colim_l s_{m,l}\colon \Sigma^2MU(m)\to MU(m+1)
$$
and so we get a sequential spectrum with spaces $MU_{2m}=MU(m)$ and $MU_{2m+1}=\Sigma MU(m)$.
Let $MU \to QMU$ be the fibrant replacement of $MU$ with spaces
$$
QMU_n := \colim_k \Omega^kMU_{n+k}.
$$

For the following definition we write $\overline a_{m,l}$ for the canonical map $MU(m,l)\to MU(m)$, $\overline b_{m,k}$ for the canonical map $\Omega^k MU(m)\to QMU_{2m-k}$, and 
$$
\Psi^k\colon \Mapp(\Sigma^kX,Y)\xto{\approx} \Mapp (X,\Omega^kY)
$$ 
for the homeomorphism 
$$\Psi^k(f) = (x\mapsto(t\mapsto f(x,t)))$$ 
of the adjunction $\Sigma^k\dashv \Omega^k$.

\begin{defn}
\label{defA}
We define 
$$
A\colon \Mapp(\Sigma^kX_+, MU(m,l))\to \Map(X,QMU_n)
$$
by
$$
A(g)=\overline b_{m,k}\circ \Psi^k(\overline{a}_{m,l}\circ g)|_X.
$$
\end{defn}

Now we assume that $X$ is a smooth manifold. 
In particular, since $X$ is finite dimensional, we get as a consequence of Freudenthal's Suspension Theorem, as stated in \cite[Corollary 3.2.3]{Kochman}, that $A$ induces a bijection on homotopy classes, provided $m,l$ are sufficiently large compared to $\dim X$. 
Hence from the perspective of homotopy theory we need only concern ourselves with the image of $A$, which we now give an alternative description of. 

\begin{defn}
Let $\Map^A(X,QMU_n)$ be the space of maps $g\colon X\to QMU_n$ such that $g=A(g')$ for some $g'\colon \Sigma^kX_+\to MU(m,l)$. 
\end{defn}
Let $g\colon \Sigma^kX_+\to MU(m,l)$. We note that if $g'\colon \Sigma^{k+2}X_+\to MU(m+1,l)$ satisfies $A(g')=A(g)$, then $g'=s_{m,l}\circ  g \times \id_\C$. 
Similarly, if $g''\colon \Sigma^{k}X_+\to MU(m,l+2)$ satisfies $A(g'')=A(g)$, then $g''=q_{m,l}\circ g$. 
Hence we consider systems of maps $\{g_{m,l}\}$ where $g_{m,l}\colon \Sigma^{k}X_+\to MU(m,l)$, with $2m=n+k$, such that for all sufficiently large $m,l$ we have 
\begin{align*} 
q_{m,l}\circ g_{m,l}=g_{m,l+1}, ~ \text{and} ~ s_{m,l}\circ \id_\C\times g_{m,l}=g_{m+1,l}.
\end{align*}
We say that two such systems $\{g_{m,l}\}$ and $\{g'_{m,l}\}$ are equivalent if $g'_{m,l}=g_{m,l}$ for all sufficiently large $m,l$. 
From this discussion we conclude:
\begin{prop}\label{FunctionsFromMapA}
There is a bijection between maps $g\in \Map^A(X,QMU_n)$ and the set of equivalence classes of systems of maps $\{g_{m,l}\}$. \qed 
\end{prop}


\subsection{Geometric fundamental forms and a new model}\label{sec:new_model_MUhs(p)}

Let $(X,F^*)\in \ManF$. 
We will now take smoothness of maps into account. 
Let $\mappsm(\Sigma^kX_+, MU(m,l))$ denote the space of pointed maps $\Sigma X_+\to MU(m,l)$ which are smooth on the preimage of $\gamma_{m,l}$.

\begin{defn}
We define $\mapsm(X,QMU_n)$ as the set of maps $g\colon X\to QMU_n$ such that $g=A(\gsm)$ for some $\gsm\in \mappsm(\Sigma^kX_+,MU(m,l))$. 
We define a map 
$$\phi^{m,l}_{sm}\colon \mappsm(\Sigma^kX_+,MU(m,l))\to \Ah^{2m}(\Sigma^k X_+;\Vh_*)_{\cl}$$ by $\gsm\mapsto \gsm^*\phi_{m,l}.$
For $n+k=2m$, applying the integration map $\int_{\Sigma^kX_+/X}$ provides a form in $\Ah^n(X;\Vh_*)$.
\end{defn}

\begin{lemma}\label{lemma:phi_ml_is_well_defined}
We have a well-defined map
\begin{align*}
\phi_{\sm}^n \colon \mapsm(X,QMU_n) \to \Ah^n(X;\Vh_*)_{\cl},
\end{align*}
given by
\[
g=A(\gsm)\mapsto \int_{\Sigma^kX_+/X}\gsm^*(\phi_{m,l}).
\] 
\end{lemma}

\begin{proof}
We will use Proposition \ref{phimlAreCompatible} to show that $\phi_\sm^n(g)$ is independent of the choice of $\gsm$ with $A(\gsm)=g$. By Proposition \ref{FunctionsFromMapA}, it suffices to show 
\begin{align}
\label{qconditionphinsm}
\int_{\Sigma^kX_+/X} \gsm^*\phi_{m,l}&=\int_{\Sigma^kX_+/X} (q_{m,l}\circ \gsm)^*\phi_{m,l+1}\\
\label{sconditionphinsm}
\int_{\Sigma^kX_+/X} \gsm^*\phi_{m,l+1} &= \int_{\Sigma^{k+2}X_+/X} (s_{m,l}\circ  \Sigma^2\gsm)^*\phi_{m+1,l}.
\end{align}
We first note that, since $q_{m,l}=\Thom(\overline{i_{m,l}})$, \eqref{qconditionphinsm} holds even before applying $\int_{\Sigma^kX_+/X}$ by Proposition \ref{phimlAreCompatible}.
Next observe that the map $s_{m,l}\circ \Sigma^2g_{\sm}$ is characterized by restricting to $\overline{j_{m,l}}\circ(\gsm|_{X\times \R^k}\times \id_{\C})$ on $(X\times \R^k)\times\C$. Proposition \ref{phimlAreCompatible} implies
\begin{align*} 
\left(\left(s_{m,l}\circ \Sigma^2\gsm\right)|_{X\times\R^k\times \C}\right)^*\phi_{m+1,l} &= \left(g_{\sm}|_{X\times \R^k}\times\id_\C\right)^*(\overline{j_{m,l}}^*\phi_{m+1,l}) \\
\nonumber &= (g_{\sm}|_{\R^k\times X})^*\phi_{m,l}\otimes U_\C
\end{align*}
and \eqref{sconditionphinsm} follows, since $\int_{X\times \C/X}U_\C =1$.
\end{proof}

\begin{remark}
Recall the $k$-th standard simplex 
\begin{equation}\label{StandardTopologicaln-simplex}
\DD^k =\{ (t_0,\ldots,t_k) \in \R^{k+1} | 0 \leq t_j \leq 1, \sum t_j =1 \}.
\end{equation} 
Note that $\DD^k$ is a smooth manifold with corners. 
Recall that a map $\DD^k \to S$ is smooth if it can be extended to a smooth map on an open neighborhood of $\DD^k$ in $\R^{k+1}$. 
We refer to \cite[\S 1.5]{wall} for any details. 
\end{remark}

\begin{remark}
\label{rem:order_of_Delta_and_X_integration} 
We note that in the following argument and in the remainder of the paper we consider the product $\DD^{\bullet} \times X$ instead of $X \times \DD^{\bullet}$. 
While the latter is more common for arguments in homotopy theory, the former has the advantage that it simplifies arguments which involve integration over $\Delta^{\bullet}$ as defined in Remark \ref{IntegrationOverTheFiber}. 
For, if $\omega$ is a form on $\DD^1 \times X$, then the derivative of the integral over the fiber satisfies the formula
\begin{align*}
d\int_{\DD^1\times X/X}\omega = -\int_{\DD^1\times X/X}d\omega +\iota_1^*\omega- \iota_0^*\omega.
\end{align*}
For a form $\omega$ on $X \times \DD^1$, however, we get  
\begin{align*}
d\int_{X\times \DD^1/X}\omega = - \int_{X\times \DD^1/X}d\omega+ (-1)^{\dim_{\R} X}(\iota_1^*\omega- \iota_0^*\omega)
\end{align*}
with an additional sign $(-1)^{\dim_{\R} X}$ which arises from a necessary reshuffling of the coordinates. 
For a complex manifold $X$, this would not matter. Since we want to allow manifolds in $\ManF$, we chose to work with $\DD^{\bullet}\times X$. 
\end{remark}

Since both the domain and the codomain of the map $\phi^n_{\sm}$ are defined for every finite-dimensional manifold, we can replace $X$ with $\Delta^k \times X$ for any $k$. 
Moreover, 
Lemma \ref{lemma:phi_ml_is_well_defined} applies to $\Delta^k \times X$ as well. 
Hence we draw the following conclusion:

\begin{prop}\label{prop:phi_sm_is_comptatible_with_suspensions}
The maps $\phi_{\sm}^{m,l}$ induce maps of simplicial sets 
\begin{align*}
\phi_{\sm}^{n} \colon \mapsm(\Delta^{\bullet} \times X, QMU_n) \to \Ah^n(\Delta^{\bullet} \times X;\Vh_*)
\end{align*}
which 
fit into commutative diagrams of the form 
\begin{align*}
\xymatrix{
\mappsm(\Sigma(\Delta^{\bullet} \times X)_+, QMU_{n+1}) \ar[r]^-{\phi_{\sm}^{n+1}} \ar[d] &  \Ah^{n+1}(\Sigma(\Delta^{\bullet} \times X)_+;\Vh_*)_{\cl} \ar[d] \\
\mappsm((\Delta^{\bullet} \times X)_+, QMU_{n}) \ar[r]^-{\phi_{\sm}^{n}} & \Ah^{n}((\Delta^{\bullet} \times X)_+;\Vh_*)_{\cl}
}
\end{align*}
where the right-hand vertical map is given by integration over the fiber. \qed
\end{prop}

From the first assertion of Proposition \ref{prop:phi_sm_is_comptatible_with_suspensions} we conclude that we have the following diagram of simplicial sets 
\begin{align}\label{eq:htpy_cartesian_MU_hs(X)}
\xymatrix{
 & \mapsm(\Delta^{\bullet} \times X,QMU_n) \ar[d]^-{\phi^n_{\sm}} \\
F^p\Ah^n(\Delta^{\bullet} \times X;\Vh_*)_{\cl} \ar[r] & 
\Ah^n(\Delta^{\bullet} \times X;\Vh_*)_{\cl}
}    
\end{align}
where $F^p\Ah^n(\Delta^{\bullet} \times X;\Vh_*)_{\cl}$ is defined in Definition \ref{def:Hodge_filtration_on_product}.

\begin{defn}\label{def:MU_hs(X)}
Let $(X,F^*) \in \ManF$ and let $n$ and $p$ be integers. 
Let $\MUhs(p)_n(X)$ be the homotopy pullback of diagram \eqref{eq:htpy_cartesian_MU_hs(X)}. 
We set 
\begin{align*}
\MUhs^n(p)(X):= \pi_0(\MUhs(p)_n(X)).    
\end{align*}
\end{defn}

Recall that we identify the interval $[0,1]$ and $\Delta^1$ via $t\leftrightarrow (t,1-t)$. Denote by $\iota_t^1$ the map
\begin{align}
\label{eq:def_of_iota_t^1}
\iota_t^1 \colon \Delta^1 \times \Delta^0 \times X \into \Delta^1 \times \Delta^1 \times X
\end{align}
with image $\Delta^1 \times \{t\} \times X$, 
and by $\iota_s^2$ the map
\begin{align}
\label{eq:def_of_iota_s^2}
\iota_s^2 \colon \Delta^0 \times \Delta^1 \times X \into \Delta^1 \times \Delta^1 \times X
\end{align}
with image $\{s\} \times \Delta^1 \times X$. 
Since \eqref{eq:htpy_cartesian_MU_hs(X)} is a diagram of Kan complexes, the set $\pi_0(\MUhs(p)_n(X))$ has the following concrete description:

\begin{lemma}\label{thm:HFC_concrete_triples} 
Let $(X,F^*)\in \ManF$. Every element in $\MUhs^n(p)(X)$ is represented by a triple 
\[
(g,\omega,h)\in \Mapsm(X,QMU_n)\times F^p\Ah^n(X;\Vh_*)_{\cl} \times \Ah^n(X\times\Delta^1;\Vh_*)_{\cl},
\] 
such that $\iota_1^*h=\phi^n_{\sm}(g)$ and $\iota_0^*h=\omega$. 
Two such triples $(g_0,\omega_0,h_0)$ and $(g_1,\omega_1,h_1)$ are homotopic if there is a triple $(\gb,\omb,h_\bullet)$ in
\[
\mapsm(\Delta^1 \times X,QMU_n) \times F^p\Ah^n(\Delta^1 \times X;\Vh_*)_{\cl} \times \Ah^n(\Delta^1 \times \Delta^1 \times X;\Vh_*)_{\cl}
\]
which satisfies $(\iota^2_1)^*\hb = \phi^n_{\sm}(\gb)$ and $(\iota^2_0)^*\hb = \omb$ in $\Ah^{n}(\Delta^1 \times X;\Vh_*)$, and   
such that $\iota_i^*(\gb,\omb,\hb)=(g_i,\omega_i,h_i)$ for $i=0,1$. 
The latter means, in particular, $(\iota^1_0)^* \hb = h_0$ and $(\iota^1_1)^*\hb = h_1$. \qed 
\end{lemma}

\begin{remark}
\label{rem:MUhs_is_a_Hodge_filtered_extension}
We note that $X \mapsto \MUhs^*(p)(X)$ is a Hodge filtered extension of $(MU^*,\phi^p)$ over $\ManF$ in the sense of Definition \ref{def:axioms}.  
The structure maps $I_\hs$ and $R_\hs$ are induced by the maps of simplicial sets $\MUhs(p)_n(X) \to \Mapsm(X\times \Delta^{\bullet},QMU_n)$ and $\MUhs(p)_n(X) \to F^p\Ah^n(\Delta^{\bullet} \times X;\Vh_*)$, respectively. 
The map $a_\hs$ arises from the fact that $\MUhs(p)_n(X)$ is homotopy equivalent to the homotopy fiber of the induced map of simplicial sets 
\[
\Mapsm(\Delta^{\bullet} \times X,QMU_n) \to \Ah^n(\Delta^{\bullet} \times X;\Vh_*)/F^p\Ah^n(\Delta^{\bullet} \times X;\Vh_*).
\]
The induced long exact sequence of the homotopy fiber yields an exact sequence of the form \eqref{eq:long_es_intro}. 
\end{remark}


\section{Comparison of homotopy models}
\label{sectionComparisonOfHFCTheories}

Now we show that for $X$ a complex manifold, there is a natural isomorphism $\MUD^*(p)(X) \cong \MUhs^*(p)(X)$ for every $p\in \Z$. 
We will show the existence of the isomorphism by showing that there is a zig-zag of weak equivalence between the defining homotopy pullbacks. 
We restrict to $X \in\Manc$, as opposed to $X\in\ManF$, since $\MUD^n(p)(X)$ has only been defined for complex manifolds in \cite{hfc}. 
We would expect, however, that an extension of $\MUD^*(p)(X)$ to $\ManF$ is possible as well. 

\subsection{Notation} 
Let $\sPre = \sPre(\Manc)$ be the category of simplicial presheaves on the site $\Manc$ with Grothendieck topology defined by open subsets. 
We consider $\sPre$ with the local projective model structure. 
The weak equivalences are maps which induce weak equivalences of simplicial sets on stalks. 
We denote the resulting homotopy category by $\hosPre$. 
Let $\sPrep$ denote the category of pointed presheaves. 
We denote the category of presheaves of sequential spectra of simplicial sets on $\Manc$ by $\mathrm{Sp}(\sPrep)$ and consider it as a model category with the model structure induced by stabilising the one of $\sPrep$. 
We denote the resulting homotopy category by $\hoSpPre$. 

For a topological space $Z$, we consider the simplicial presheaf $\Sing Z$ on $\Manc$ 
whose $n$-simplices are continuous maps
\[
\DD^n \times X \to Z.
\]
Note that we have a canonical isomorphism of simplicial sets $\sing\,Z = \Sing Z(\pt)$. 
For every $CW$-complex $Z$, the simplicial presheaf $\Sing Z$ is objectwise fibrant and satisfies hypercover descent in $\Manc$ by \cite[Theorem 1.3]{di} (see also \cite[Lemma 2.3]{hfc}). 
A continuous map $Z \to Z'$ induces a map of presheaves $\Sing Z \to \Sing Z'$. 
If $E$ is a sequential spectrum of topological spaces, the structure maps of $E$ turn $\Sing E$ into a presheaf of spectra of simplicial sets on  $\Manc$. 


\subsection{\texorpdfstring{$\MUhs(p)$}{MUhsp} is a presheaf of spectra}
\label{sec:MUhs_is_a_presheaf_of_spectra}

First we observe that $\Mapsm(\DD^{\bullet} \times X,QMU)$ is a simplicial spectrum with structure maps defined as follows: 
Recall that, as described in \cite[page 379]{hs}, a $k$-simplex of the simplicial loop space $\Omega^{\simp}A_{\bullet}$ of a simplicial set $A_{\bullet}$ 
can be described as a sequence 
$$
a_0,\ldots,a_{k}\in A_{k+1}
$$
satisfying the conditions 
\begin{align}
\label{eq:simplicialLoopSpaceIdentities_A_bullet} \partial_i^*a_i &= \partial_i^*a_{i-1}\\
\nonumber \partial_0^*a_0 &=*=\partial_{k+1}^*a_{k}
\end{align} 
The homeomorphism $QMU_n \xto{\approx} \Omega QMU_{n+1}$ induces for each $n$ a natural isomorphism $\Mapsm(\DD^{\bullet} \times X,QMU_n) \xto{\cong} \Mapsm(\DD^{\bullet} \times X,\Omega QMU_{n+1})$.
The adjunction between the suspension and loop space functors then induces a natural isomorphism 
$\Mappsm(\Sigma(\DD^{\bullet} \times X)_+,QMU_{n+1}) \xto{\cong} \Mapsm(\DD^{\bullet} \times X,QMU_{n+1})$. 
A pointed map $\Sigma(\DD^{\bullet} \times X)_+ \to QMU_{n+1}$ corresponds to a map $\Delta^1 \times \DD^{\bullet} \times X \to QMU_{n+1}$ which restricts to the constant map on both subspaces $\{0\} \times \DD^{\bullet} \times X$ and $\{1\} \times \DD^{\bullet} \times X$ with value the canonical basepoint of $QMU_{n+1}$. 
The restriction of any such map to the $(k+1)$-simplices in the standard triangulation of $\DD^1 \times \Delta^k$ leads to a sequence of maps 
\[
g_0,\ldots,g_{k} \colon \Delta^{k+1} \times X \to QMU_{n+1}, 
\] 
i.e., a sequence of $(k+1)$-simplices in $\Mapsm(\DD^{\bullet} \times X,QMU_{n+1})$ satisfying the relations corresponding to \eqref{eq:simplicialLoopSpaceIdentities_A_bullet}.  
This defines a natural map of simplicial sets 
\begin{align}
\label{eq:Mapsm_to_Omega}
\Mapsm(\DD^{\bullet} \times X,QMU_{n}) \to \Omega^{\simp}\Mapsm(\DD^{\bullet} \times X,QMU_{n+1}) 
\end{align}
which provides the sequence $n \mapsto \Mapsm(\DD^{\bullet} \times X,QMU_{n})$ with the structure of a sequential spectrum of simplicial sets.

Second, we choose a concrete model for presheaves of Eilenberg--MacLane spaces. 
More concretely, we will prove:
\begin{prop}\label{DoldKanA(V)}
The simplicial presheaf $X\mapsto \Ah^{n}(\Delta^{\bullet} \times X;\Vh_*)$ on $\Man$ is weakly equivalent to $K(\Ah^*(\Vh_*),n)$.
\end{prop}
We recall that the Dold--Kan correspondence is an equivalence of categories between simplicial abelian groups and connective chain-complexes under which weak equivalences correspond to quasi-isomorphisms. 
It associates to a simplicial presheaf $F_\bullet$ the normalized chain complex, $N(F_\bullet)$. 
There is also the Moore complex of a simplicial presheaf, $M(F_\bullet)$, which has $n$-th presheaf $F_n$, and differentials given by $\sum(-1)^i\partial_i$, where $\partial_i$ are the face maps of $F_\bullet$. 
There is a natural map $N(F_\bullet)\to M(F_\bullet)$ which is a quasi-isomorphism by \cite[Theorem III.2.1]{GoerssJardine}. 
Proposition \ref{DoldKanA(V)} therefore follows from the Lemma \ref{tauInverseToIntegration} below.

\begin{lemma}\label{tauInverseToIntegration}
Integration over the fiber $\int_{\DD^k \times X/X}$ is a chain homotopy equivalence 
$$
\xymatrix@C=2em{
\cdots\ar[r]^-{\partial^*}& \Ah^n(\DD^k \times X;\Vh_*)_{\cl}\ar[r]^-{\partial^*}\ar[d]_{\int_{\Delta^k\times X/X}} & \cdots \ar[r]^-{\partial^*}& \Ah^n(\Delta^1\times X;\Vh_*)_{\cl}\ar[r]^-{\partial^*}\ar[d]_{\int_{\Delta^1\times X/X}}& \Ah^n(X;\Vh_*)_{\cl}\ar[d]_{\id} \\
\cdots\ar[r]^-d &\Ah^{n-k}(X;\Vh_*)\ar[r]^-d & \cdots \ar[r]^-d & \Ah^{n-1}(X;\Vh_*)\ar[r]^-d & \Ah^{n}(X;\Vh_*)_{\cl},}
$$ 
with homotopy inverse $\tau$ defined below.
\end{lemma}
Lemma \ref{tauInverseToIntegration} is essentially \cite[Corollary D.14]{hs}. However, there the homotopy inverse is not given, and so we will give its construction. 
Let $v_i\in\Delta^k$ be the point with $t_i=1$. Let $p_k\colon \Delta^k\backslash v_0\to \Delta^{k-1}$ be radial projection onto the $0$-th face, 
$$
p_k(t_0,\dots, t_k)=\left(\frac{t_1}{1-t_0},\dots, \frac{t_k}{1-t_0}\right).
$$
Let $g\colon [0,1]\to \mathbb R$ be smooth, and vanishing in a neighborhood of $0$ and equaling $1$ in a neighborhood of $1$. Recall that we identify $\Delta^1$ with the unit interval $[0,1]$ via $t\leftrightarrow (t,1-t)$.
We can then define  
\begin{align*}
    h_{k}\colon \Ah^n(\Delta^{k-1} \times X;\Vh_*) & \to \Ah^n(\Delta^{k} \times X;\Vh_*) \\
        \omega &\mapsto g(1-t_0)\cdot (p_k\times \id)^*\omega.
\end{align*}
The map $h$ is a contraction of the complex $\Ah^n(\DD^{\bullet} \times X;\Vh_*)$. 
We define $\tau$ recursively by $\tau_0=\id\colon \Ah^n(X;\Vh_*)_{\cl}\to\Ah^n(X;\Vh_*)_{\cl}$, and for $k>0$,
\begin{align}\label{deftau}
    \tau_k\colon \Ah^{n-k}(X;\Vh_*) &\to \Ah^n(\DD^k \times X;\Vh_*)_{\cl}\\
    \nonumber \tau_k &= d\circ h_{k}\circ \tau_{k-1}.
\end{align}

Let $\iota_t\colon X\to \Delta^1 \times X$, $\iota_t(x)=(t,x)$, be the inclusion at $t$ map. 
Then we have  
\begin{equation}
\label{equationiotatau}
    \iota_0^*\circ \tau_1 = d\quad \text{and}\quad \iota_1^*\circ\tau_1  =0.
\end{equation}
Using the techniques of \cite[Appendix D]{hs}, one can also prove: 

\begin{prop}\label{FpEilenbergMacLane}
Integration over the fiber $\int_{\DD^k \times X/X}$ restrict to a chain homotopy equivalence
$$
\xymatrix@C=1.5em{
\cdots\ar[r]^-{\partial^*}& F^p\Ah^n(\DD^k \times X;\Vh_*)_{\cl}\ar[r]^-{\partial^*}\ar[d] & \cdots \ar[r]^-{\partial^*}& F^p\Ah^n(\DD^1 \times X;\Vh_*)_{\cl}\ar[r]^-{\partial^*}\ar[d]& F^p\Ah^n(X;\Vh_*)_{\cl}\ar[d] \\
\cdots\ar[r]^-d &F^p\Ah^{n-k}(X;\Vh_*)\ar[r]^-d & \cdots \ar[r]^-d & F^p\Ah^{n-1}(X;\Vh_*)\ar[r]^-d & F^p\Ah^{n}(X;\Vh_*)_{\cl}
}
$$
with homotopy inverse, the restriction of $\tau$.
\end{prop}


Now we can turn this into a concrete model for the presheaf of Eilenberg--MacLane spectra as follows. 
As pointed out in \cite[page 381]{hs}, and analogous to the construction of the map \eqref{eq:Mapsm_to_Omega}, restriction to the standard triangulation of $\DD^1 \times \Delta^k$ yields a weak equivalence of simplicial sets 
\begin{align*}
s_\bullet ^{\Ah} \colon 
\Ah^{n}(\DD^{\bullet} \times X;\Vh_*)_{\cl} \to \Omega^{\simp}\Ah^{n+1}(\DD^{\bullet} \times X;\Vh_*)_{\cl}.
\end{align*}

We now introduce the notation $\Ahs^*(\Vh_*)$ and show that we obtain  $\Omega$-spectra: 

\begin{prop}\label{prop:A_Sigma_is_a_spectrum}
The sequences 
\[
\Ahs^*(\Vh_*) \colon n \mapsto \Ah^{n}(\DD^{\bullet} \times -;\Vh_*)_{\cl}
\]
and, for every integer $p$, 
\[
F^p\Ahs^*(\Vh_*) \colon n \mapsto F^p\Ah^{n}(\DD^{\bullet} \times -;\Vh_*)_{\cl}
\]  
define $\Omega$-spectra in the category of sequential spectra of presheaves on $\Manc$. 
\end{prop}
\begin{proof}
The maps $s_\bullet ^{\Ah}$ provide the objectwise structure maps of the first spectrum. 
As the filtration on forms on $X$ is independent of the simplicial identities which define the map, 
the maps $s_\bullet ^{\Ah}$ restrict to natural weak equivalences 
\begin{align*}
s_\bullet ^{F^p\Ah} \colon 
F^p\Ah^{n}(\DD^{\bullet} \times X;\Vh_*)_{\cl} \to \Omega^{\simp}F^p\Ah^{n+1}(\DD^{\bullet} \times X;\Vh_*)_{\cl}.
\end{align*}
Since $\Ah^{n}(\DD^{\bullet} \times X;\Vh_*)_{\cl}$ is a simplicial abelian group and hence a Kan complex for each $n$, the sequence $n \mapsto \Ah^{n}(\DD^{\bullet} \times X;\Vh_*)_{\cl}$ is an $\Omega$-spectrum of simplicial sets for every $X$. 
Moreover, since each of the simplicial presheaves $\Ah^{n}(\DD^{\bullet} \times -;\Vh_*)_{\cl}$ and $F^p\Ah^{n}(\DD^{\bullet} \times -;\Vh_*)_{\cl}$ satisfy descent with respect to hypercovers, we obtain fibrant objects in the local projective model structure of presheaves of sequential spectra on $\Manc$ by \cite[Corollary 7.1]{dhi} (see also \cite{compact}).  
This proves the assertion. 
\end{proof}


Next we check that the maps $\phi^n_{\sm}$ induce a natural map of simplicial spectra 
\begin{align*}
\Mapsm(\DD^{\bullet} \times X,QMU) \xto{\phi_{\sm}} \Ahs^*(X,\Vh_*). 
\end{align*}
Since the maps $\phi^n_{\sm}$ are natural in $\DD^{\bullet} \times X$, a sequence $g_0,\ldots, g_{k}$ of $(k+1)$-simplices in $\Mapsm(\DD^{\bullet} \times X,QMU_{n+1})$ satisfying the relations corresponding to \eqref{eq:simplicialLoopSpaceIdentities_A_bullet}   induces a sequence $\phi^{n+1}_{\sm}(g_0),\ldots, \phi^{n+1}_{\sm}(g_{k})$ of $(k+1)$-simplices in $\Ah^{n+1}(\DD^{\bullet} \times X;\Vh_*)_{\cl}$ satisfying similar relations. 
Thus $\phi^n_{\sm}$ induces a natural map on the simplicial loop spaces as well. 
In fact, we get a commutative diagram of the form 
\begin{align*}
\xymatrix{
\Mapsm(\DD^{\bullet} \times X,QMU_{n}) \ar[r]^-{\phi^n_{\sm}} \ar[d] & \Ah^{n}(\DD^{\bullet} \times X;\Vh_*)_{\cl} \ar[d] \\
\Omega^{\simp}\Mapsm(\DD^{\bullet} \times X,QMU_{n+1}) \ar[r]^-{\phi^{n+1}_{\sm}} & \Omega^{\simp}\Ah^{n+1}(\DD^{\bullet} \times X;\Vh_*)_{\cl},
}
\end{align*}
since both vertical maps arise from the restriction to the standard triangulation of $\Delta^1 \times \Delta^k$. 


\subsection{Comparison of homotopy models}
\label{subsec:comparison_of_htpy_models}

Now we construct the comparison map.

The map $MU\to QMU$ induces a map $\Sing(MU)\to\Sing(QMU)$. 
We then obtain a map $\sing(MU)\to \Sing(QMU)$ by precomposing with the isomorphism of simplicial sets $\sing(MU)\cong \Sing(MU)(\pt)$. 
It follows from \cite[Proposition 3.11]{hfc} that this map induces an isomorphism on stalks and so is a weak equivalence.  
Let $\Mapsm(-\times \Delta^{\bullet},QMU) \to \Sing(QMU)$ be the map of presheaves of spectra induced by forgetting smoothness. This is an objectwise weak equivalence, since every continuous map is homotopic to a smooth map by Whitney's approximation theorem. 
We also note that this is a map between fibrant objects by \cite[Lemma 3.12]{hfc}. 
Let $H(\Ah^*(\Vh_*)) \to \Ahs^*(\Vh_*)$ be the map of presheaves of spectra which for the $n$-th spaces is given by the map $\tau$ defined in Lemma \ref{tauInverseToIntegration}. 
Its homotopy inverse is induced by integrating over the fiber. 
Similarly, let $H(F^p\Ah^*(\Vh_*)) \to F^p\Ahs^*(\Vh_*)$ be the map of presheaves of spectra induced by the restriction of $\tau$. Again, this map has a homotopy inverse which is induced by integrating over the fiber. 
In total, we have a diagram of presheaves of spectra 
\begin{align}\label{eq:map_of_diagrams_for_htpy_pb}
\xymatrix{
H(F^p\Ah^*(\Vh_*)) \ar[r] \ar@/^/[dd]^-{\simeq} & H(\Ah^*(\Vh_*)) \ar@/^/[dd]^-{\simeq} & \ar[l]_-{\phi} \sing(MU) \ar[d]^-{\simeq} \\
 & & \Sing(QMU) \\
F^p\Ahs^*(\Vh_*) \ar@/^/[uu]^-{\simeq} \ar[r] & \Ahs^*(\Vh_*) \ar@/^/[uu]^-{\simeq} & \ar[l]_-{\phi_{\sm}} \Mapsm(\DD^{\bullet} \times -,QMU) \ar[u]_-{\simeq}.
}    
\end{align}

\begin{defn}\label{def:MUhs(p)}
We write $\MUhs(p)$ for the homotopy of the bottom row of diagram \eqref{eq:map_of_diagrams_for_htpy_pb}. 
\end{defn}

\begin{remark}\label{rem:holom_vs_smooth_forms} 
Recall that the homotopy pullback of the top row of diagram \eqref{eq:map_of_diagrams_for_htpy_pb} is $\MUD(p)$ by definition. 
We note that in \cite[\S 4]{hfc} the complex of holomorphic forms $\Omega^*(X;\Vh_*)$ is used instead of smooth forms to define $\MUD(p)$. 
However, since the canonical maps $\Omega^*(\Vh_*) \to \Ah^*(\Vh_*)$ and $\Omega^{*\geqslant p}(\Vh_*)\to F^p\Ah^*(\Vh_*)$ are quasi-isomorphisms the homotopy pullback of \eqref{eq:map_of_diagrams_for_htpy_pb} represents Hodge filtered cobordism groups which are canonically isomorphic to the ones of \cite[Definition 4.2]{hfc}.  
\end{remark}

\begin{theorem}\label{thm:htpy_pullbacks_MUD_MUhs_are_we}
The homotopy pullbacks $\MUD(p)$ and $\MUhs(p)$ 
are isomorphic in the homotopy category $\hoSpPre$ of presheaves of spectra on $\Manc$. 
\end{theorem}
\begin{proof}  
The assertion is a consequence of the observation on homotopy pullbacks in model categories formulated in the following lemma.
\end{proof}

\begin{lemma}\label{lemma:htpy_pullback_lemma}
Let $\Ch$ be a proper simplicial model category. 
We consider the following diagram 
\begin{align}\label{eq:map_of_diagrams_for_htpy_pb_lemma}
\xymatrix{
C_1 \ar[r] \ar@/^/[dd]^-{\simeq} & B_1  \ar@/^/[dd]^-{\simeq} & \ar[l] A_1\ar[dr]^-{\simeq} & \\
 & & & A_0 \\
C_2 \ar@/^/[uu]^-{\simeq} \ar[r] & B_2 \ar@/^/[uu]^-{\simeq} & \ar[l] A_2 \ar[ur]_{\simeq} & 
}    
\end{align}
in which all vertical maps are weak equivalences and in which the left-hand squares commute. 
We also assume that $A_0$ and $A_2$ are fibrant. 
Then the homotopy pullbacks $C_1\times^h_{B_1}A_1$ and $C_2\times^h_{B_2}A_2$ of the top and bottom row, respectively, are weakly equivalent.  
\end{lemma}
\begin{proof}
First we take the homotopy pullback $A_1\times^h_{A_0}A_2$ of the right-hand vertical maps. 
The two induced maps $A_1\times^h_{A_0}A_2 \to A_1$ and $A_1\times^h_{A_0}A_2 \to A_2$ are weak equivalences. 
Hence the induced maps on homotopy pullbacks  \begin{align*}
C_1\times^h_{B_1}(A_1\times^h_{A_0}A_2) \to C_1\times^h_{B_1}A_1 ~ \text{and} ~ 
C_2\times^h_{B_2}(A_1\times^h_{A_0}A_2) \to C_2\times^h_{B_2}A_2
\end{align*}
are weak equivalences. 
Thus it remains to observe that the homotopy pullbacks $C_1\times^h_{B_1}(A_1\times^h_{A_0}A_2)$ and $C_2\times^h_{B_2}(A_1\times^h_{A_0}A_2)$ are weakly equivalent. 
This follows from the following diagram 
\begin{align*}
    \xymatrix{
C_1 \ar[r] \ar@/^/[d]^-{\simeq} & B_1  \ar@/^/[d]^-{\simeq} & \ar[l] A_1\times^h_{A_0}A_2 \ar@{=}[d] \\
C_2 \ar@/^/[u]^-{\simeq} \ar[r] & B_2 \ar@/^/[u]^-{\simeq} & \ar[l] A_1\times^h_{A_0}A_2 
}    
\end{align*}
in which the right-hand square commutes up to homotopy and the vertical maps are weak equivalences. 
\end{proof}

As a consequence of Theorem  \ref{thm:htpy_pullbacks_MUD_MUhs_are_we} we get that the respective cohomology groups represented by the two homotopy pullbacks are isomorphic. 
This enables us to prove the following result: 

\begin{theorem}\label{thm:map_MUD_to_MUhs_is_an_iso}
Let $X$ be a complex manifold and $n,p$ be integers. 
Then there is a natural isomorphism of groups  
\begin{align*}
\MUD^n(p)(X) \cong \MUhs^n(p)(X).
\end{align*}
\end{theorem}
\begin{proof}
By Theorem \ref{thm:htpy_pullbacks_MUD_MUhs_are_we} it remains to relate the groups $\MUhs^n(p)(X)$ of Definition \ref{def:MU_hs(X)} to $\MUhs(p)$-cohomology. 
Each of the presheaves of spectra in the bottom row of diagram \eqref{eq:map_of_diagrams_for_htpy_pb} satisfy  levelwise hypercover descent and the structure maps are objectwise weak equivalences. Hence each of the presheaves of spectra in the bottom row is an $\Omega$-spectrum. 
Hence the $n$-th space $\MUhs(p)_n$ of $\MUhs(p)$ represents $\MUhs(p)$-cohomology, i.e., there is a natural isomorphism
\begin{align*}
\Hom_{\hoSpPre}(\Sigma^{\infty}(X_+),\Sigma^n\MUhs(p)) 
\cong \Hom_{\hosPre}(X,\MUhs(p)_n).
\end{align*}
Moreover, we can compute the simplicial presheaf  $\MUhs(p)_n$ levelwise as the homotopy pullback of the $n$th spaces of the presheaves of spectra in the bottom row of \eqref{eq:map_of_diagrams_for_htpy_pb}. 
By \cite[Proposition 2.7]{aj}, we can compute this homotopy pullback objectwise in the sense that there is a natural isomorphism 
\begin{align*}
\Hom_{\hosPre}(X,\MUhs(p)_n) \cong \pi_0(\MUhs(p)_n(X)). 
\end{align*}
Since we have $\MUhs^n(p)(X) = \pi_0(\MUhs(p)_n(X))$ by definition of the left hand side, this proves the assertion of the theorem. 
\end{proof}

As we pointed out in section \ref{sec:axioms}, 
the functor $\MUD(p)$ is a Hodge filtered cohomology extension over $(MU^*,\phi^p)$ on $\Manc$. 
We conclude this section with the observation that the isomorphism of Theorem \ref{thm:map_MUD_to_MUhs_is_an_iso} is one of Hodge filtered extensions. 

\begin{remark}
\label{rem:MUD_MUhs_iso_is_Hodge_filtered}
Let $I_\Dh$, $R_\Dh$ and $a_\Dh$ denote the structure maps of $\MUD(p)$ as a Hodge filtered extension of $(MU^*,\phi^p)$ in the sense of Definition \ref{def:axioms} and Example \ref{ex:ED_is_a_HF_extension}.  
We observed in Remark \ref{rem:MUhs_is_a_Hodge_filtered_extension} that $\MUhs^*(p)$ is a Hodge filtered extension over $(MU^*,\phi^p)$. 
Now it suffices to observe that the zig-zag of weak equivalences we constructed between $\MUD(p)$ and $\MUhs(p)$ actually is a zig-zag between the homotopy fiber sequences that induce the respective structure maps. 
Hence we conclude that 
the isomorphism $\MUD^*(p)(X)\cong \MUhs^*(p)(X)$ of Theorem \ref{thm:map_MUD_to_MUhs_is_an_iso} is an isomorphism of Hodge filtered extensions of $(MU^*,\phi^p)$ over $\Manc$.
\end{remark}


\section{From homotopy to geometry}
\label{section:from_htpy_to_geom}

We are now going to use the Pontryagin--Thom construction to define a natural isomorphism of  groups
\[
\kappa \colon \MUhs^n(p)(X) \to MU^n(p)(X)
\]
for every $n$ and $p$ and every $X \in \ManF$.

\subsection{Representatives in the homotopy pullback}
 
First we will represent elements of  $\MUhs^n(p)(X)$ in a suitable way. 
Recall the map $A$ from Definition \ref{defA}. 
Let 
\[
\Mapf(X,QMU_n)
\]
denote the set of maps $g\colon X\to QMU_n$ such that $g=A(\gp)$ for some map $\gp\colon \Sigma^k X_+\to MU(m,l)$ which  is smooth on $\gp^{-1}(\gamma_{m,l})$, and transverse to the $0$-section $\iota_{m,l}\colon \Grml\to \gamma_{m,l}$.

\begin{theorem}\label{thm:HFC_concrete_triples_transversal}
Let $(X,F^*)\in \ManF$, and let $n$ and $p$ be integers. 
For every element $\gamma \in \MUhs^n(p)(X)$, there is a representative $(g, \omega, h)$ as in  Lemma \ref{thm:HFC_concrete_triples} with $g \in \Mapf(X, QMU_n)$.
\end{theorem}

To prove Theorem \ref{thm:HFC_concrete_triples_transversal} we are going to use the following 
general fact about homotopy pullbacks. 
Since it is important for the later arguments, we provide a detailed proof. 

\begin{lemma}\label{lemma:homotopy_induces_homotopy_on_forms}

Let $g_{\bullet} \colon \Delta^1 \times X \to QMU_n$ be a homotopy between $g_0=\iota_0^*\gb$ and $g_1=\iota_1^*\gb$. 
Assume we have a triple $(g_0,\omega,h_0)$ which represents an element in $\MUhs^n(p)(X)$. 
Then there is a form $h_1 \in \Ah^n(\Delta^1 \times X;\Vh_*)$ such that 
the triple $(g_1,\omega,h_1)$ is homotopic to $(g_0,\omega,h_0)$.   
\end{lemma}

\begin{proof} 
The pullback along the projection $\pi_X \colon \Delta^1 \times X \to X$ of $\phi_{\sm}(g_0)$ yields a closed form  $\pi_X^*\phi^n_{\sm}(g_0) \in \Ah^{n}(\Delta^1 \times X;\Vh_*)_{\cl}$ which is constant on $\Delta^1$. 
We set $h_1 := h_0 + \phi_{\sm}(\gb) - \pi_X^*\phi_{\sm}(g_0)$. 
The restrictions along $\iota_t \colon \Delta^0 \times X \into \Delta^1 \times X$ for $t=0$ and $t=1$, respectively, yield  
\begin{align*}
\iota_1^*h_1 & = \iota_1^*h_0 + \iota_1^*\phi_{\sm}(\gb) - \iota_1^*\pi_X^*\phi_{\sm}(g_0) \\
& = \phi_{\sm}(g_0) + \phi_{\sm}(g_1) - \phi_{\sm}(g_0) \\
& = \phi_{\sm}(g_1) 
\end{align*}
and
\begin{align*}
\iota_0^*h_1 & = \iota_0^*h_0 + \iota_0^*\phi_{\sm}(\gb) - \iota_0^*\pi_X^*\phi_{\sm}(g_0) \\
& = \omega + \phi_{\sm}(g_0) - \phi_{\sm}(g_0) \\
& = \omega.
\end{align*}
Hence the triple $(g_1,\omega, h_1)$ represents an element in $\MUhs^n(p)(X)$.   
Now we construct a homotopy between $(g_0,\omega,h_0)$ and $(g_1,\omega,h_1)$. 
The homotopies $\gb$ and $\omega_{\bullet}:=\pi_X^*\omega$ satisfy the requirements of Lemma \ref{thm:HFC_concrete_triples}. 
It remains to find a compatible homotopy $\hb$. 
To construct $\hb$ we consider the map 
\[
G_{\bullet} \colon \Delta^1 \times \Delta^1 \times X \to QMU_n
\]
defined by $G_s(t,x) = g_{st}(x)$. 
We think of $G_{\bullet}$ as a homotopy between the maps $G_0 \colon (t,x)\mapsto g_0(x)$ and $G_1 \colon (t,x) \mapsto g_t(x)$.  
We set 
\[
\hb := \pi_{\Delta^1 \times X}^*(h_0 -\pi_X^*\phi_{\sm}(g_0)) + \phi_{\sm}(G_{\bullet}) \in \Ah^n(\Delta^1 \times \Delta^1 \times X;\Vh_*)_{\cl}  
\]
where $\pi_{\Delta^1 \times X} \colon \DD^1 \times \DD^1 \times X \to \DD^1 \times X$ denotes the projection to the two right hand factors. 
Next we compute the pullbacks along the various inclusions into the copies of $\Delta^1$. 
Recall the map $\iota_s^2$ from \eqref{eq:def_of_iota_s^2}. 
The restriction to $1$ on the left most factor in $\Delta^1 \times \Delta^1 \times X$ yields 
\begin{align*}
(\iota^2_1)^*\hb & = 
(\iota^2_1)^*\pi_{\Delta^1 \times X}^*(h_0 -\pi_X^*\phi_{\sm}(g_0)) + (\iota^2_1)^*\phi_{\sm}(G_{\bullet}) \\
& = h_0 - \pi_X^*\phi_{\sm}(g_0) + \phi_{\sm}(\gb)  \\
& = h_1.
\end{align*}
The restriction to $0$ yields 
\begin{align*}
(\iota^2_0)^*\hb & = 
(\iota^2_0)^*\pi_{\Delta^1 \times X}^*(h_0-\pi_X^*\phi_{\sm}(g_0)) + (\iota^2_0)^*\phi_{\sm}(G_{\bullet}) \\
& = h_0  - \pi_X^*\phi_{\sm}(g_0) + \phi_{\sm}(g_{0\cdot \bullet}) \\
& = h_0 - \pi_X^*\phi_{\sm}(g_0) + \pi_X^*\phi_{\sm}(g_0) \\
& = h_0.
\end{align*}
%
%
Now recall the map $\iota_t^1$ from \eqref{eq:def_of_iota_t^1}.
The restriction to $1$ on the middle factor in $\Delta^1 \times \Delta^1 \times X$ yields 
\begin{align*}
(\iota^1_1)^*\hb & = 
(\iota^1_1)^*\pi_{\Delta^1 \times X}^*(h_0-\pi_X^*\phi_{\sm}(g_0)) + (\iota^1_1)^*\phi_{\sm}(G_{\bullet}) \\
& =  \iota_1^*(h_0-\pi_X^*\phi_{\sm}(g_0)) + \phi_{\sm}(\gb) \\
& = \phi_{\sm}(g_0) - \phi_{\sm}(g_0) + \phi_{\sm}(\gb)  \\
& = \phi_{\sm}(\gb). 
\end{align*}
The restriction to $0$ yields 
\begin{align*}
(\iota^1_0)^*\hb & = 
(\iota^1_0)^*\pi_{\Delta^1 \times X}^*(h_0-\pi_X^*\phi_{\sm}(g_0)) + (\iota^1_0)^*\phi_{\sm}(G_{\bullet}) \\
& = \iota_0^*(h_0 -\pi_X^*\phi_{\sm}(g_0)) + \phi_{\sm}(g_{0\cdot \bullet})  \\
& =  \omega - \phi_{\sm}(g_0) + \phi_{\sm}(g_0) \\
& = \omega.  
\end{align*} 
Thus the triple $(\gb,\pi_X^*\omega,\hb)$ is a homotopy between $(g_0,\omega,h_0)$ and $(g_1,\omega, h_1)$.  
\end{proof}

\begin{proof}[Proof of Theorem \ref{thm:HFC_concrete_triples_transversal}]
Let $\gamma \in \MUD^n(p)(X)$ and $(g, \omega, h)$ be a representative as in Lemma  \ref{thm:HFC_concrete_triples}. 
By Freudenthal's Suspension Theorem there is a map $g'\colon \Sigma^kX_+\to MU(m,l)$ such that $A(g')$ is homotopic to $g$. Since smooth maps are dense among continuous maps, we can by Thom's transversality theorem find a map $g_\pitchfork \colon \Sigma^kX_+\to MU(m,l)$ which is smooth on $g_\pitchfork^{-1}(\gamma_{m,l})$, homotopic to $g'$ and transverse to $\iota_{m,l}$. 
Then $A(\gp)\in\Mapf(X,QMU_n)$ is homotopic to $g$.
By Lemma \ref{lemma:homotopy_induces_homotopy_on_forms}, there exists an $h_1$ such that $(A(\gp),\omega,h_1)$ represents an element in $\MUhs^n(p)(X)$ and a homotopy between $(g,\omega,h)$ and $(A(\gp),\omega,h_1)$.  
Hence the latter too represents the class $\gamma$ in $\MUhs^n(p)(X)$. 
\end{proof}


\begin{defn}\label{def:MUD_transversal_subset}
We denote by $\MUhs^{\pitchfork}(p)_n(X)_0$ the subset of triples $(g,\omega,h)$ in $\MUhs(p)_n(X)_0$ such that 
$g\in \Mapf(X,QMU_n)$.
\end{defn}


\subsection{A geometric Pontryagin--Thom map}

We will now define a geometric Pontryagin--Thom map 
$$
\rho_\nabla\colon \Mapf(X,QMU_n)\to \geocycles^n(X).
$$
Given $\gp\colon \Sigma^kX_+\to MU(m,l)$ which is smooth on 
$$
U_\gp=(\gp)^{-1}(\gamma_{m,l}) \subset X\times \R^k  \subset \Sigma^kX_+,
$$ 
and transverse to $\iota_{m,l}$, we consider  
the following commutative diagram: 
\begin{equation}
\label{diagramPontryaginThomExplained}
\xymatrix{
Z\ar@/_2.0pc/[ddd]_{f_{\gp}}\ar[r]^-{\gp|_Z}\ar[dd]_-{i} & \Gr_m\left(\C^{m+l}\right) \ar[d]_-{\iota_{m,l}} 
& \\
 & \gamma_{m,l} \ar[d]_-{
 } &  \\
U_{\gp} \ar[ur]^-{\gp|_{U_\gp}} \ar[d]_-{\pi} \ar[r] & MU(m,l) \ar[r]^-{
} & QMU_{2m} \\
X & & 
}
\end{equation}
The square is cartesian, $Z = (\gp)^{-1}(\Gr_m(\C^{m+l}))$, and $\pi$ is the restriction to $U_\gp$ of the projection $X\times \R^k\to X$. 
The derivative of $\gp$ induces a real isomorphism
\[
D\gp\colon N(i)\to \left(\gp|_Z\right)^*N(\iota_{m,l})=:N_\gp
\]
where $N(i)$ is the normal bundle of $i$.
This gives $f_\gp=\pi\circ i$ a complex orientation. It is easy to check that $f_\gp$ is proper. 
We get a connection on $N_\gp$ by   
$$
\nabla_{\gp}:=\left(\gp|_Z\right)^*\nabla_{m,l}.
$$
We then define 
$$
\rho_{\nabla,m,l}(\gp)=(f_\gp, N_\gp, \nabla_\gp)\in \geocycles^{n}(X),
$$
and for $g=A(\gp)$, we define
$$
\rho_\nabla(g) = \rho_{\nabla,m,l}(\gp) \in \geocycles^n(X).
$$

\begin{prop}\label{prop:rho_nabla_is_well_defined}
The map $\rhog\colon \Mapf(X,QMU_n)\to \geocycles^n(X)$ is well-defined.
\end{prop}

\begin{proof}
Let $\gp\in \Mappf(\Sigma^kX_+,MU(m,l))$. By Theorem \ref{FunctionsFromMapA} it suffices to show
\begin{align*}
\rho_{\nabla,k,m,l}(\gp)&= \rho_{\nabla,k,m,l+1}(q_{m,l}\circ \gp) \\ 
\rho_{\nabla,k,m,l}(\gp)&= \rho_{\nabla,k+2,m+1,l}(s_{m,l}\circ \Sigma^2\gp).
\end{align*}
We use the notation of diagram \eqref{diagramPontryaginThomExplained}. The first identity follows from the following commutative diagram, where all squares are cartesian: 
\[
\xymatrix{
  & Z_{q_{m,l}\circ \gp} \ar[d]\ar[dr] & \\
Z_{\gp}\ar[d]\ar[dr]\ar[ur] & \Gr_{m}(\C^{m+l+1})\ar[dr]|\hole_(.3){\iota_{m,l+1}} & U_\gp \ar[d] 
\ar@/^3pc/[dd] \\
\Gr_m(\C^{m+l})\ar[dr]_{\iota_{m,l}}\ar[ur]|\hole & U_\gp\ar[d]_-(.3){\gp}\ar[ur]_(.3){\id} \ar[dr] & \mathbb \gamma_{m,l+1} \\
 &  \gamma_{m,l} \ar[ur]|\hole_(.27){\overline{i_{m,l}}} & X.
}
\]
For the second identity, we first note that $s_{m,l}\circ\Sigma^2\gp$ is characterized by
$$
\left(s_{m,l}\circ \Sigma^2\gp\right)|_{X\times \R^k\times \C} = \overline{j_{m,l}}\circ(\left(\gp|_{X\times \R^k}\right)\times\id_\C
).
$$
It is clear that $(\gp|_{X\times\R^k})\times\id_\C $ is transverse to $\{0\}\times \Gr_m(\C^{m+l})$, with inverse image $\{0\}\times Z_{\gp}$. We then consider the following commutative diagram, where all squares are cartesian squares of manifolds:  
$$
\xymatrix{
  & Z_{s_{m,l}\circ \Sigma^2\gp} \ar[d]\ar[dr] & \\
\{0\}\times Z_{\gp}\ar[d]\ar[dr]\ar[ur] & \Gr_{m+1}(\C^{m+1+l})\ar[dr]|\hole & \C\times U_\gp \ar[d]_{} \ar@/^3pc/[dd] \\
\Gr_m(\C^{m+l})\ar[dr]_{0\oplus \iota_{m,l}}\ar[ur]|\hole & \C\times U_\gp\ar[d]_-(.35){\gp\times \id_\C}\ar[ur]_(.3){\id} \ar[dr] & \gamma_{m+1,l} \\
 & \trivbc\oplus \gamma_{m,l} \ar[ur]|\hole_(.3){\overline{j_{m,l}}} & X.
}
$$
It is now clear that the map $\{0\}\times Z_{\gp}\to Z_{s_{m,l}\circ \Sigma^2\gp}$ is an isomorphism of geometric complex oriented maps over $X$. This proves the assertion.  
\end{proof}

In order to better understand the current $\phi(\rhog(g))$, we define fundamental currents by 
\begin{align} \label{geometric_fundamental_current}
\phi_{\nabla_{m,l}}=(\iota_{m,l})_*K(\nabla_{m,l}) \in \Ds^{2m}(MU(m,l);\Vh_*). 
\end{align}
The cartesian square of diagram \eqref{diagramPontryaginThomExplained} yields by Theorem \ref{currentialPushPull} the formula of currents
$$
\left(\gp|_{U_\gp}\right)^*\phi_{\nabla_{m,l}}=i_*K(\nabla_\gp).
$$
Pushing this down to $X$, we obtain an expression for $\phi(\rhog(g))$ in terms of $\phi_{\nabla_{m,l}}$:
\begin{equation} \label{fundamentalEquationforFundamentalCurrents}
    \phi(\rhog(A(\gp)))=(f_\gp)_*K(\nabla_\gp) = \pi_*\left(\gp|_{U_\gp}\right)^*\phi_{\nabla_{m,l}}.
\end{equation}

\begin{remark} \label{compatabilityOfGeometricFundamentalCurrents}
We now have a map
\begin{align*}
\phi\circ\rhog=:\phi_\nabla &\colon \Mappf(X,QMU_n)\to \Ds^n(X;\Vh_*)
\end{align*}
of presheaves on the site of manifolds and local diffeomorphisms. Applying this with $X=\gamma_{m,l}$ and $g$ the canonical inclusion $\gamma_{m,l}\to MU(m,l)=\gamma_{m,l}\sqcup\{\infty\}$, we deduce that $\overline{i_{m,l}}^*\phi_{\nabla_{m,l+1}}= \phi_{\nabla_{m,l}}$ and $\left(\pi_{\gamma_{m,l}}^{\gamma_{m,l}\times\C}\right)_*\overline{j_{m,l}}^*\phi_{\nabla_{m+1,l}}= \phi_{\nabla_{m,l}}$.
\end{remark}

We will now compare the fundamental forms $\phi_{m,l}$ of \eqref{definitionPhiml} with the fundamental currents $\phi_{\nabla_{m,l}}$ of \eqref{geometric_fundamental_current}.
The Thom isomorphism is induced by integration over the fiber, which we have defined to be the restriction and co-restriction of the currential pushforward. 
Since $\int_{MU(m,l)/Gr_m(\C^{m+l})}U_{m,l}=1$, we have  
$$
\int_{MU(m,l)/\Gr_m(\C^{m,l})} \pi_{m,l}^*K(\nabla_{m,l})\wedge U_{m,l} = K(\nabla_{m,l})
$$
by the projection formula for the currential pushforward $f_*(T\wedge f^*\omega)=f_*(T)\wedge \omega$. 
We also have
$$
(\pi_{m,l})_*\left((\iota_{m,l})_* K(\nabla_{m,l})\right)= K(\nabla_{m,l})
$$
since $\pi_{m,l}\circ \iota_{m,l} = \id_{\Gr_m(\C^{m+l})}$. 
This establishes the identity  
\begin{align}\label{eq:comparison_of_fund_currents_in_cohomology}
[{\phi_{m,l}}]=[\phi_{\nabla_{m,l}}] ~\text{in} ~ H^{2m}(MU(m,l);\Vh_*).
\end{align}
It follows that there are currents
$$
\alpha_{m,l} \in \Ds^{2m-1}(MU(m,l);\Vh_*)
$$ 
such that 
\begin{align}\label{eq:alpha_defining_equation_as_current_on_MU(m,l)}
d\alpha_{m,l} =\phi_{\nabla_{m,l}}- \phi_{m,l} ~ \text{in} ~ 
\Ds^{2m}(MU(m,l);\Vh_*).
\end{align}
Since $WF(d\alpha_{m,l})=WF(\phi_{\nabla_{m,l}})\subset N(\iota_{m,l})$, we can by Lemma \ref{4.11} furthermore assume 
$$
WF(\alpha)\subset N(\iota_{m,l}).
$$
This implies, in particular, that $g^*\alpha_{m,l}$ is defined whenever $g\pitchfork \iota_{m,l}$. 
\begin{remark}\label{remark:alphamlWellDefined}
Suppose $\alpha'\in \Ds^{2m-1}(MU(m,l);\Vh_*)$ also satisfies $d\alpha'  =\phi_{\nabla_{m,l}}- \phi_{m,l}$.
We get 
$$
d(\alpha_{m,l}-\alpha')=0.
$$ 
Since $H^{2m-1}(MU(m,l);\Vh_*)=0$ we can find a current $\beta$ in  $\Ds^{2m-2}(MU(m,l);\Vh_*)$ such that $d\beta = \alpha_{m,l}-\alpha'$. 
We will eventually only need $\alpha_{m,l}$ modulo exact currents, and as such we see that $\alpha_{m,l}$ is well-defined.
\end{remark}

From Proposition \ref{MathaiQuillenThomFormsCompatible} and Remarks \ref{compatabilityOfGeometricFundamentalCurrents} and \ref{remark:alphamlWellDefined} it follows that modulo exact currents, we have
$$
\overline{i_{m,l}}^*\alpha_{m,l+1}=\alpha_{m,l}\quad \text{ and }\quad  \left(\pi_{m,l}\right)_*(\overline{j_{m,l}}^*\alpha_{m+1,l})=\alpha_{m,l},
$$
where $\pi_{m,l}\colon \trivbc\oplus \gamma_{m,l}\to \gamma_{m,l}$ is the projection. 
Hence by Proposition \ref{FunctionsFromMapA} we may define a map 
\begin{align*}
\pbalpha\colon \Mapf(X,QMU_n)  \to \Ds^*(X;\Vh_*)/\Imm(d)   
\end{align*}
by the setting 
\begin{align*}
\pbalpha(g) :=(-1)^k\pi_*\gp^*\alpha_{m,l}
\end{align*}
for any $\gp\in \Mapp^\pitchfork(\Sigma^kX_+,MU(m,l))$ with $A(\gp)=g$, where $\pi\colon X\times\R^k\to X$ is the projection as in diagram \eqref{diagramPontryaginThomExplained}.
The sign is here to counteract the sign appearing as $d$ passes through $\pi_*$ in the computation
\begin{align}
    \nonumber d\pbalpha(g) &= (-1)^kd\pi_*\gp^*(\alpha_{m,l}) \\
    \label{dpbalphaeq} &= \pi_*d\gp^*(\alpha_{m,l}) \\
    \nonumber &= \pi_* \gp^*(\phi_{\nabla_{m,l}}-\phi_{m,l})\\
    \nonumber &= \phi(\rhog(g))-\phi_\sm^n(g),
\end{align}
where the last equality uses  \eqref{fundamentalEquationforFundamentalCurrents}.


\subsection{Construction of the map \texorpdfstring{$\kappa$}{kappa}}
\label{subsec:construction_of_kappa}


For the following definition recall the notation $\MUhs^{\pitchfork}(p)_n(X)_0$ from Definition \ref{def:MUD_transversal_subset} and 
the group $ZMU^n(p)(X)$ of Hodge filtered cobordism cycles from Definition \ref{def:ZMUn(p)(X)}.

\begin{defn}\label{def:kappa_on_set_of_transversal_maps}
We define the map 
\begin{align*}
\kappa \colon \MUhs^{\pitchfork}(p)_n(X)_0 \to ZMU^n(p)(X)
\end{align*}
by 
\[
(g,\omega,h)\mapsto \left(\rhog(g),\ \omega ,\ \pbalpha(g) + \int_{\Delta^1 \times X/X} h\right).
\]
\end{defn}

\begin{lemma}\label{lemma:we_do_get_a_HFC_cycle}
For $(g,\omega,h)\in \MUhs^{\pitchfork}(p)_n(X)_0$ 
we have 
\[
\kappa(g,\omega,h)\in ZMU^n(p)(X).
\]
\end{lemma}
\begin{proof}
We work within the context of diagram \eqref{diagramPontryaginThomExplained}, within which we defined both $\rhog$ and $\pbalpha$. 
Since $(g,\omega,h)$ has the form described in Lemma \ref{thm:HFC_concrete_triples} we have 
$$
d\int_{\Delta^1 \times X/X}h = \iota_1^*h-\iota_0^*h = \phi_\sm^n(g)-\omega.
$$ Using \eqref{dpbalphaeq} we compute
\begin{align*}
    R(\kappa(g,\omega,h)) &= \phi\left(\rhog(g)\right)-d\pbalpha(g)-d\int_{\Delta^1 \times X/X}h \\
    &= \phi(\rhog(g))-\big(\phi(\rhog(g))-\phi_\sm(g)\big)-(\phi_\sm^n(g)-\omega)\\
    &= \omega.
\end{align*}
Since $\omega \in F^p\Ah^n(X;\Vh_*)$, this shows that $\kappa$ indeed is a map to $ZMU^n(p)(X)$. 
\end{proof}

Our next goal is to prove that $\kappa$ takes homotopic triples to cobordant Hodge filtered cycles. We first establish two lemmas:
\begin{lemma}
\label{lemma:htpy_formula}
Let $X\in\ManF$, and let 
$\iota_t\colon X\to \R\times X$ be given by $\iota_t(x)=(t,x)$. 
Then for $\hfcycle\in MU^n(p)(\R\times X)$ we have
$$
\iota_1^*\hfcycle-\iota_0^*\hfcycle=a\left[\int_{[0,1]\times X/X} R(\hfcycle)\right]
$$
in $MU^n(p)(X)$.
\end{lemma}

\begin{remark} 
This lemma, and the ensuing proof, is inspired by \cite[Lemma 5.1]{Bunke2010}.
As in differential cohomology, this is a general result applying to all Hodge filtered cohomology theories. 
\end{remark}
\begin{remark}
The corresponding formula for differential cohomology provides a formula relating $f_1^*(\gamma)$ to $f_0^*(\gamma)$ for homotopic maps $f_1,f_0\colon X\to Y$. To get that implication in our context, we need the homotopy $f_\bullet$ to preserve filtrations, i.e., we would need $f_\bullet^* F^p\Ah^*(Y)\subset F^p\Ah^*(\R\times X)$. 
This is a severe restriction on permissible homotopies, and certainly it is not always the case that homotopic holomorphic maps can be connected by such a homotopy. 
In fact, it is not even always possible to connect two homotopic holomorphic maps by a homotopy through holomorphic maps.
\end{remark}

\begin{proof}[Proof of Lemma \ref{lemma:htpy_formula}]
Let $\pi\colon \R \times X\to X$ denote the projection, and let $\beta\in \Ah^{n-1}(\R \times X;\Vh_*)$ be a form with $d\beta\in F^p\Ah^n(\R\times X;\Vh_*)$ for which
$$
\hfcycle = \pi^*(\iota_0^*(\gamma)) + a[\beta].
$$
Note the identity
\begin{equation}\label{equationinproofofhtpyformula}
    \iota_1^*\gamma-\iota_0^*\gamma=a[\iota_1^*\beta-\iota_0^*\beta].
\end{equation}
We also have 
$$
R(\hfcycle)=R(\pi^*(\iota_0^*(\gamma))) + [d\beta],
$$ 
from which we get
\begin{align*}
    \int_{[0,1]\times X/X} R(\gamma) &= \int_{[0,1]\times X/X} d\beta \\
        &= \iota_1^*\beta-\iota_0^*\beta\quad \mod \ \Imm d.
\end{align*}
Together with \eqref{equationinproofofhtpyformula} this finishes the proof.
\end{proof}

Let $q \colon Y\to X$ be a smooth map. 
We denote by  
\begin{align}
\label{eq:def_Map_pitchfork,F}
\Map_*^{\pitchfork,q}\left(\Sigma^kX_+,MU(m,l)\right) \subset \Mappf\left(\Sigma^kX_+,MU(m,l)\right)
\end{align} 
the subset of maps $\gp$ such that $\gp \circ \Sigma^k q \in \Mappf\left(\Sigma^kY_+,MU(m,l)\right)$. 

\begin{lemma}
 \label{PontryaginThomIsNatural}
$\rho_\nabla$ is natural in the sense that for $q \colon Y\to X$ a smooth map, and $\gp\in \Map_*^{\pitchfork,q}\left(\Sigma^kX_+,MU(m,l)\right)$ we have: 
$$
\rhog (A(\gp\circ \Sigma^k q_+)) = q^*\rho_\nabla(A(\gp))\in \geocycles^n(Y). 
$$
\end{lemma}
\begin{proof}
It is well-known that given smooth maps $f_1 \colon M_1\to M_2$ and $f_2 \colon M_2\to M_3$ such that $f_2\pitchfork S\subset M_3$, we have $f_1\pitchfork (f_2^{-1}(S))$ if and only if $(f_2\circ f_1) \pitchfork S$. From this we see that $q^*$ is defined on $\rhog(g)$.
Now the assertion follows from the following cartesian diagram
\begin{align*}
\xymatrix{
Z_Y \ar[d]_-{i_Y} \ar[r]^-{} & Z_X \ar[r]^-{\gp|_{Z_X}}\ar[d]_-{i_X} & \Gr_m\left(\C^{m+l}\right)\ar[d]^-{\iota_{m,l}}\\
U_Y \ar[d]_-{\pi_Y} \ar[r]^-{\id \times q} & U_X \ar[d]^-{\pi_X} \ar[r]^-{\gp|_{U_X}} & \gamma_{m,l} \\
Y \ar[r]^-q & X & 
}
\end{align*}
with $Z_X = Z_{g} = (g)^{-1}(\Gr_m(\C^{m+l}))$, $Z_Y = (q\times \id)^{-1}Z_X$ and
\begin{align*}
U_X&=g^{-1}(\gamma_{m,l}) \subset X \times\R^k \subset \Sigma^kX_+ \\
U_Y&=(q\times \id)^{-1}(U_X) \subset Y \times \R^k \subset \Sigma^kY_+. \qedhere
\end{align*}
\end{proof}

\begin{defn}\label{def:MUD_F_transversal_subset} 
We denote by $\MUhs^{\pitchfork,q}(p)_n(X)_0$ the subset of triples $(g,\omega,h)$ in $\MUhs^{\pitchfork}(p)_n(X)_0$ such that $g=A(\gp)$ for a map $\gp$ in $\Map_*^{\pitchfork,q}\left(\Sigma^k X_+, MU(m,l)\right)$. 
\end{defn}

Using the notation of Definition \ref{def:MUD_F_transversal_subset} the pullback for $\MUhs$ along a morphism $q \colon Y\to X$ of $\ManF$ is induced by the map
\begin{align*}
q^* \colon \MUhs^{\pitchfork,q}(p)_n(X)_0 &\to \MUhs^{\pitchfork}(p)_n(Y)_0\\
\left(A(\gp),\omega,h\right) &\mapsto \left(A(\gp\circ\Sigma^k q_+),q^*\omega,(\id_{\Delta^1} \times q)^*h\right).
\end{align*}

\begin{lemma}\label{lemma:HFC_concrete_triples_transversal_to_F}
For $(g, \omega, h)$ in $\MUhs^{\pitchfork,q}(p)_n(X)_0$, the following diagram commutes:
\begin{align*}
\xymatrix{
\MUhs^{\pitchfork,q}(p)_n(X)_0 \ar[r]^-{q^*} \ar[d]_-{\kappa} & \MUhs^{\pitchfork}(p)_n(Y)_0 \ar[d]^-{\kappa} \\
ZMU_q^n(p)(X)\ar[r]_{q^*} & ZMU^n(p)(Y).
}
\end{align*}
\end{lemma}

\begin{proof}
Let $g=A(\gp)$ for $\gp\colon \Sigma^k X\to MU(m,l)$, and write $\left(\pi_X^{\Sigma X}\right)_*$ for the pushforward map $\Ds^*(\Sigma X)\to \Ds^*(X)$. Using Lemma \ref{PontryaginThomIsNatural} and the push-pull formula for currents of Theorem \ref{currentialPushPull}, we compute:
\begin{align*}
    q^*\kappa(g,\omega,h) &= q^*\left(\rhog(g), \omega, \pbalpha(g)+\int_{\Delta^1 \times X/X}h\right)\\
    &=\left(q^*\rhog(g), q^*\omega, q^*\left( \left(\pi_X^{\Sigma^k X}\right)_*\gp^*(\alpha_{m,l})+\int _{\Delta^1 \times X/X}h\right)\right)\\
    &= \Bigg(\rho_{\nabla}(A(\gp\circ\Sigma^k q_+)), q^*\omega, \\
    &\quad\quad\quad\quad\quad\quad\quad \left(\pi_Y^{\Sigma^k Y}\right)_*(\gp\circ\Sigma^k q_+)^*\alpha_{m,l}+\int_{\Delta^1\times Y/Y}(\id_{\Delta^1} \times q)^*h\Bigg)\\
    &=\kappa\left(A(\gp\circ\Sigma^k q_+), q^*\omega, (\id_{\Delta^1} \times q)^*h\right). \qedhere
\end{align*}
\end{proof}

We are now ready to show that homotopic triples yield cobordant Hodge filtered cycles. We again work within the context of diagram \eqref{diagramPontryaginThomExplained}.
\begin{lemma}\label{lemma:HFC_homotopies_yield_geometric_bordisms}
Let $(g_0,\omega_0,h_0)$ and $(g_1,\omega_1,h_1)$ be triples in $\MUhs^{\pitchfork}(p)_n(X)_0$ and assume there is a homotopy $(\gb, \omega_\bullet,\hb)$ between them. 
Then    
\[
\kappa(g_1,\omega_1,h_1)-\kappa(g_0,\omega_0,h_0)\in BMU^n(p)(X)
\]    
where $BMU^n(p)(X)$ is defined in \eqref{eq:def_of_geometric_bordism_data}. 
\end{lemma}


\begin{proof}  
By Lemma \ref{thm:HFC_concrete_triples} we can assume that
$$
(\gb,\omega_\bullet, \hb)\in \MUhs^\pitchfork(p)_n(\Delta^1 \times X)_0,
$$
and that $g_t(x)$ is constant as a function of $t$ in a neighborhood of each endpoint. We may then extend $(\gb,\omega_\bullet, \hb)$ to an element of $\MUhs^\pitchfork(p)_n(\R\times X)_0$, placing $(g_\bullet, \omega_\bullet, h_\bullet)$ in the domain of $\kappa$.
We get 
$$
\kappa(\gb,\omega_\bullet, \hb)\in ZMU^n(p)(\R \times X).
$$ 
Let $\iota_t\colon X\to \R \times X$ be inclusion at $t$. The pullback $\iota_t^*(\gb,\omega_\bullet, \hb)\in  \MUhs^\pitchfork(p)_n(X)_0$ is defined for $t=0,1$. By naturality of $\kappa$, we have 
$$
\iota_t^*\kappa(\gb,\omega_\bullet, \hb) = \kappa(g_t,\omega_t,h_t)\in ZMU^n(p)(X).
$$
Hence we get
$$
\kappa(g_1,\omega_1,h_1)-\kappa(g_0,\omega_0,h_0) = (\iota_1^*-\iota_0^*)(\kappa(\gb,\omega_\bullet,\hb)).
$$
By Lemma \ref{lemma:htpy_formula} we get
$$
[\kappa(g_1,\omega_1,h_1)-\kappa(g_0,\omega_0,h_0)] = a\left[\int_{\Delta^1 \times X/X} \omb \right].
$$
Since $\omb\in F^p\Ah^n(\Delta^1 \times X;\Vh_*)$, we have $\int_{\Delta^1 \times X/X}\omb \in F^p\Ah^{n-1}(X;\Vh_*)$. Since $a\left[F^p\Ah^{n-1}(X;\Vh_*)\right]$ is a subset of $BMU^n(p)(X)$, this proves the assertion.  
\end{proof}


Every element in $\MUhs^n(p)(X)$ can be represented by a triple in the subset $\MUhs^{\pitchfork,q}(p)_n(X)_0$. This follows from Thom's transversality theorem and Lemma \ref{lemma:homotopy_induces_homotopy_on_forms}.
Combining this with Lemmas \ref{lemma:HFC_concrete_triples_transversal_to_F} and \ref{lemma:HFC_homotopies_yield_geometric_bordisms}, we have proven the following result: 

\begin{theorem}\label{thm:kappa_commutes_with_pullback}
Let $q\colon Y \to X$ be a morphism in $\ManF$ and $p$ be an integer. 
Then we have $q^* \circ \kappa = \kappa \circ q_{\hs}^*$ where $q^*$  denotes the pullback in $MU^*(p)(-)$ and $q^*_{\hs}$ the pullback in $\MUhs^*(p)(-)$. 
\end{theorem}


\subsection{The map \texorpdfstring{$\kappa$}{kappa} is an isomorphism}
\label{subsec:kappa_is_an_iso}

We  will now show that $\kappa$ is a homomorphism and that it respects the structure maps of Hodge filtered cohomology theories. The respective long exact sequences of the two theories will then imply that $\kappa$ is an isomorphism. 

The addition on $\MUhs^n(p)(X)$ is induced by the following binary operation on $\MUhs^{\pitchfork}(p)_n(X)_0$;
$$
(A(g_1),\omega_1,h_1)+(A(g_2),\omega_2,h_2)=(A((g_1\vee g_2)\circ \pinch), \omega_1+\omega_2, h_1+h_2),
$$
for maps $g_1,\ g_2\in\Mappf(\Sigma^kX_+,MU(m,l))$.
Here $\pinch\colon \Sigma^kX_+\to \Sigma^kX_+\vee\Sigma^kX_+$ is the pinch map which collapses $X\times \R^{k-1}\times\{0\}$. We observe that 
$$
\rhog\colon \Mapf(X,QMU_n) \to \geocycles^n(X)
$$
is a homomorphism in the sense that 
$$
\rhog(A((g_1\vee g_2)\circ \pinch)) = \rhog(A(g_1))+\rhog(A(g_2)).
$$
This is evident upon noting the diffeomorphism
$$
(g_1\vee g_2)^{-1}(\gamma_{m,l}) \cong g_1^{-1}\gamma_{m,l}\bigsqcup g_2^{-1}(\gamma_{m,l}).
$$ 
We can deduce the following result:

\begin{lemma}\label{lemma:kappaishomomorphism}
The map  $\kappa \colon \MUhs^n(p)(X) \to MU^n(p)(X)$ is a group homomorphism. \qed
\end{lemma}

Next we will show that $\kappa$ respects the structure maps. We lift the structure maps $I_{\hs}$, $R_{\hs}$ and $a_{\hs}$ of $\MUhs(p)$ as a Hodge filtered theory from the level of maps in the homotopy category to the level of $0$-simplices of the simplicial mapping space using Lemma \ref{thm:HFC_concrete_triples} as follows. 
\begin{align*}
    I_\hs\colon \MUhs(p)_n(X)_0 &\to \mapsm(X,QMU_n)\\
    (g,\omega,h) &\mapsto g  \\
     R_\hs\colon \MUhs(p)_n(X)_0 &\to F^p\Ah^n(X;\Vh_*)\\
     (g,\omega,h) &\mapsto \omega, \\
      a_\hs\colon d^{-1}\left(F^p\Ah^n(X;\Vh_*)\right)^{n-1} &\to \MUhs(p)_n(X)_0 \\
     h&\mapsto (0,dh, \tau_1 h)
\end{align*}
where $\tau_1$ is the map $\tau_1=d\circ h_0\colon \Ah^{n-1}(X;\Vh_*)\to \Ah^n(\Delta^1 \times X; \Vh_*)_{\cl}$ 
defined in \eqref{deftau}. 
Recall from \eqref{equationiotatau} that we have $i_0^*\tau_1 h=dh$ and $i_1^*\tau_1 h=0$, so that, using Lemma \ref{thm:HFC_concrete_triples},  $a_\hs(h)=(0,dh,\tau_1 h)$ represents a class in $\MUhs^n(p)(X)$.
We have a Pontryagin--Thom map
\begin{align*}
\rho\colon \Mapf\left(X,QMU_n\right) \to ZMU^n(X),
\end{align*}
defined by $\rho=I\circ \rhog$, i.e.,  $\rho(\gp)=(f_\gp,N_{\gp})$.
We denote the induced map $MU_h^n(X)\to MU^n(X)$ also by $\rho$. 
By Thom's transversality theorem, the homotopy classes in $\Mapf\left(X,QMU_n\right)\subset \Map(X,QMU_n)$ are all of $MU_h^n(X)$, 
where we use the subscript $h$ to indicate that we mean the homotopy-theoretic $MU$-cohomology group of homotopy classes of maps $X\to QMU_n$, as opposed to Quillen's geometric $MU$-cohomology groups, recalled in section \ref{Section:GeometricHFCBordism}. It is not hard to prove that $\rho$ is an isomorphism.

\begin{lemma}\label{lemma:kappa_respects_structure_maps}
The map $\kappa\colon \MUhs^n(p)(X)\to MU^n(p)(X)$ respects the structure maps:
\begin{align*}
    \kappa\circ a_\hs &= a,\\
    I\circ \kappa &= \rho \circ I_\hs \text{ and}\\
    R\circ \kappa &= R_\hs.
\end{align*}
\end{lemma}

\begin{proof}
Let $[h]\in H^{n-1}\left(X;\frac{\Ah^*}{F^p}(\Vh_*)\right)$ be a class represented by a an element $h \in \Ah^{n-1}(X;\Vh_*)$, with $dh\in F^p\Ah^n(X;\Vh_*)$.
We have $\int_{\Delta^1 \times X/X} \tau_1h=h$ by Lemma \ref{tauInverseToIntegration}. 
By the definition of $\kappa$ in Definition \ref{def:kappa_on_set_of_transversal_maps} we get 
\begin{align*}
\kappa \circ a_\hs([h]) & = 
a[h].
\end{align*}
 Compatibility with $I$ holds by our choice of isomorphism $\rho\colon MU_h^n(X)\cong MU^n(X)$.  
Finally, the equality
\begin{align*}
R\circ \kappa[(g,\omega,h)] =[\omega ]
\end{align*}
was verified on the level of forms in the proof of Lemma \ref{lemma:we_do_get_a_HFC_cycle}. 
\end{proof}

\begin{theorem}\label{thm:kappa_is_an_isomorphism}
For all integers $n$ and $p$ and every $(X,F^*)\in \ManF$, the homomorphism $\kappa \colon \MUhs(p)^n(X)\to MU^n(p)(X)$ is an isomorphism of groups. 
\end{theorem}
\begin{proof}
It follows from Lemma \ref{lemma:kappa_respects_structure_maps} that $\kappa$ fits into a morphism of long exact sequences:
\[
\xymatrix{
\cdots \ar[r] & H^{n-1}(X;\frac{\Ah^*}{F^p}(\Vh_*)) \ar[r]^-{a_\hs}\ar[d]^-{\id} & \MUhs^n(p)(X)\ar[r]^-{I_\hs} \ar[d]^{\kappa} & MU_h^n(X)\ar[d]^{\rho}\ar[r] & \cdots \\
\cdots \ar[r] & H^{n-1}(X;\frac{\Ah^*}{F^p}(\Vh_*)) \ar[r]^-a & MU^n(p)(X)\ar[r]^-{I}  & MU^n(X)\ar[r] & \cdots 
}
\]
Since the outer vertical maps are isomorphisms for each $n$ and $p$, the five-lemma implies that $\kappa$ is an isomorphism. 
\end{proof}

Together with Theorem \ref{thm:map_MUD_to_MUhs_is_an_iso} 
this finishes the proof of Theorem \ref{thm:mainTheoremKappaintro}.

\begin{remark}
In fact, the proof of Theorem  \ref{thm:kappa_is_an_isomorphism} shows that $\kappa$ is an isomorphism of Hodge filtered extensions over $(MU^*,\phi^p)$ in the sense of Definition \ref{def:axioms}. 
Together with Remark \ref{rem:MUD_MUhs_iso_is_Hodge_filtered} this shows that the natural isomorphism $\MUD^*(p)(X) \cong MU^*(p)(X)$ is an isomorphism of Hodge filtered extensions over $(MU^*,\phi^p)$ on $\Manc$. 
\end{remark}

\begin{remark}
Recall from \cite{hfc} and section \ref{subsec:products} that by taking direct sums over all $n,p\in \Z$, both $\MUD^*(*)(X)$ and $MU^*(*)(X)$ become bigraded rings. 
We are optimistic that one can show that the isomorphism $\MUD^*(*)(X) \cong MU^*(*)(X)$ is actually an isomorphism of bigraded rings. 
\end{remark}


\bibliographystyle{amsalpha}

\begin{thebibliography}{999999}
%

\bibitem{Benoist} O.\,Benoist, Steenrod operations and algebraic classes, preprint (2022), arXiv:2209.03685. 

\bibitem{BGV} N.\,Berline, E.\,Getzler, M.\,Vergne, Heat kernels and Dirac operators, Grundlehren Text Editions. Springer-Verlag, Berlin, 2004. 






\bibitem{Bunke2009} U.\,Bunke, T.\,Schick, I.\,Schröder, M.\,Wiethaup, Landweber exact formal group laws and smooth cohomology theories, Algebr.\,Geom.\,Topol. 9 (2009), no.\,3, 1751--1790.

\bibitem{Bunke2010} U.\,Bunke, T.\,Schick, Uniqueness of smooth extensions of generalized cohomology theories, J.\,Topol. 3 (2010), no.\,1, 110--156. 

\bibitem{CheegerSimons} J.\,Cheeger, J.\,Simons, 
Differential characters and geometric invariants. Geometry and topology (College Park, Md., 1983/84), 50--80,
Lecture Notes in Math., 1167, Springer, Berlin, 1985.

\bibitem{Hodge2} P.\,Deligne, Th\'eorie de Hodge II, Pub.\,Math.\,IHES 40 (1971), 5--57.


\bibitem{deRham} G.\,de Rham, Differentiable Manifolds, Springer-Verlag, 1984. 


\bibitem{di} D.\,Dugger, D.\,C.\,Isaksen, Topological hypercovers and $\A^1$-realizations, Math.\,Z. 246 (2004), 667--689.

\bibitem{dhi} D.\,Dugger, S.\,Hollander, D.\,C.\,Isaksen, Hypercovers and simplicial presheaves, Math.\,Proc.\,Cambridge Philos. Soc. 136 (2004), no.\,1, 9--51.



\bibitem{Esnault} H.\,Esnault, Characteristic classes of flat bundles, Topology 27 (1988), 323--352.



\bibitem{Forstneric} F.\,Forstnerič, Stein Manifolds and Holomorphic Mappings, Ergebnisse der Mathematik und ihrer Grenzgebiete, 3.\,Folge, A series of Modern Surveys in Mathematics 56, Springer, Cham (2017).

\bibitem{Freed} D.\,Freed, Dirac charge quantization and generalized differential cohomology, Surveys in differential geometry, 129--194,
Surv.\,Differ.\,Geom., 7, Int.\,Press, Somerville, MA, 2000. 

\bibitem{FreedHopkins} D.\,S.\,Freed, M.\,J.\, Hopkins, Chern--Weil forms and abstract homotopy theory, Bull.\,Amer.\,Math.\,Soc. 50 (2013), no.\,3, 431--468.

\bibitem{FreedLott} D.\,S.\,Freed, J.\,Lott, An index theorem in differential $K$-theory, Geom.\,Topol. 14 (2010), no.\,2, 903--966.




\bibitem{GoerssJardine} P.\,G.\,Goerss, J.\, F.\,Jardine, Simplicial Homotopy Theory, Progress in Mathematics, Birkhäuser Verlag, 1999. 

\bibitem{Grady-Sati} D.\,Grady, H.\,Sati,  Differential cohomotopy versus differential cohomology for M-theory and differential lifts of Postnikov towers, J.\,Geom.\,Phys. 165 (2021), Paper No. 104203, 24 pp.


\bibitem{Gromov} M.\,Gromov, Oka's Principle for Holomorphic Sections of Elliptic Bundles, J.\,Amer.\,Math. Soc. 2 (1989), no. 4, 851--897.


\bibitem{haus} K.\,B.\,Haus, Geometric Hodge filtered complex cobordism, doctoral thesis at NTNU, available at https://ntnuopen.ntnu.no/ntnu-xmlui/handle/11250/3017489. 

\bibitem{hausquick} K.\,B.\,Haus, G.\,Quick, Geometric pushforward in Hodge filtered complex cobordism and secondary invariants, preprint (2023), arXiv:2303.15899v1. 



\bibitem{HBJ} F.\,Hirzebruch, T.\,Berger, R.\,Jung, Manifolds and Modular Forms, Aspects of Mathematics Vol. E 20, Springer Fachmedien Wiesbaden, 1994. 

\bibitem{Hoer} L.\,Hörmander, The analysis of linear partial differential operators. I. Distrigution theory and Fourier analysis, Classics in Math., Springer, Berlin 2003. Reprint of the second 1990 edition. 

\bibitem{HoerIII} L.\,Hörmander, The analysis of linear partial differential operators. III. Pseudo-Differential Operators, Classics in Math., Springer, Berlin 2003. Reprint of the second 1990 edition. 

\bibitem{hfc} M.\,J.\,Hopkins, G.\,Quick, Hodge filtered complex bordism, J.\,Topology 8 (2015), 147--183. 

\bibitem{hs} M.\,J.\,Hopkins, I.\,M.\,Singer, Quadratic functions in geometry, topology, and $M$-theory, J.\,Differential Geom. 70 (2005), no.\,3, 329--452.

\bibitem{karoubi43} M.\,Karoubi, Th\'eorie g\'en\'erale des classes caract\'eristiques secondaires, K-Theory 4 (1990), no.\,1, 55--87. 

\bibitem{Kochman} S.\,O.\,Kochman, Bordism, Stable Homotopy and Adams Spectral Sequence, American Mathematical Society, 1996. 


\bibitem{Lar1} F.\,L{\'a}russon, Excision for simplicial sheaves on the Stein site and Gromov's Oka principle, Internat.\,J.\,Math. 14 (2003), no.\,2, 191--209.







\bibitem{mathai-quillen} V.\,Mathai, D.\,G.\,Quillen, Superconnections, Thom classes, and equivariant differential forms, Topology 25 (1986), 85--110. 



\bibitem{NR1} M.\,S.\,Narasimhan, S.\,Ramanan, Existence of universal connections, Amer.\,J.\,Math. 83 (1961), 563--572.

\bibitem{NR2} M.\,S.\,Narasimhan, S.\,Ramanan, Existence of universal connections, II, Amer.\,J.\,Math. 85 (1963), 223--231. 

\bibitem{compact} G.\,Quick, Homotopy theory of smooth compactifications of algebraic varieties, New York J.\,Math. 19 (2013), 533--544.

\bibitem{aj} G.\,Quick, An Abel--Jacobi invariant for cobordant cycles, Doc.\,Math. 21 (2016), 1645--1668. 


\bibitem{quillen} D.\,G.\,Quillen, Some elementary proofs of some results of complex cobordism theory using Steenrod operations, Adv.\,Math. 7 (1971), 29--56.

\bibitem{satischreiber} H.\,Sati, U.\,Schreiber, M/F-Theory as $Mf$-Theory, preprint, arXiv:2103.01877v1. 

\bibitem{SSbis} J.\,Simons, D.\,Sullivan, Axiomatic characterization of ordinary differential cohomology, J.\,Topol. 1 (2008), no.\,1, 45--56. 

\bibitem{SS} J.\,Simons, D.\,Sullivan, Structured vector bundles define differential K-theory. Quanta of maths, 579--599, Clay Math. Proc., 11, Amer.\,Math.\,Soc., Providence, RI, 2010.

\bibitem{shipley} B.\,Shipley, $H\Z$-algebra spectra are differential graded algebras, Amer.\,J.\,Math. 129 (2007), no.\,2, 351--379.




\bibitem{wall} C.\,T.\,C.\,Wall, Differential topology, Cambridge Studies in Advanced Mathematics, 156. Cambridge University Press, Cambridge, 2016.

\bibitem{Witten} E.\,Witten, Five-brane effective action in $M$-theory, J.\,Geom.\,Phys. 22 (2) (1997), 103--133.

\end{thebibliography}

\end{document}